\documentclass[a4paper,11pt]{amsart}
\usepackage{amsfonts,amsthm,amsmath,amssymb,graphicx}
\theoremstyle{plain}
\newtheorem{lem}{Lemma}[section]
\newtheorem{prop}[lem]{Proposition}
\newtheorem{thm}[lem]{Theorem}
\newtheorem{cor}[lem]{Corollary}
\theoremstyle{definition}

\theoremstyle{remark}

\DeclareMathOperator{\rank}{rank}

\DeclareMathOperator{\sym}{sym}
\DeclareMathOperator{\spn}{span}

\DeclareMathOperator{\mult}{mult}

\DeclareMathOperator{\gen}{gen}

\newcommand{\bmu}{\boldsymbol \mu}
\newcommand{\bdelta}{\boldsymbol \delta}
\newcommand{\bbeta}{\boldsymbol \beta}
\newcommand{\calC}{{\mathcal C}}

\newcommand{\Z}{\mathbb Z}
\newcommand{\Q}{\mathbb Q}
\newcommand{\A}{\mathbb A}
\newcommand{\F}{\mathbb F}
\newcommand{\E}{\mathbb E}
\newcommand{\R}{\mathbb R}
\newcommand{\C}{\mathbb C}
\newcommand{\stufe}{\mathcal N }

\newcommand{\Eis}{\mathcal E}
\newcommand{\h}{\mathcal H}
\newcommand{\K}{\mathcal K}
\newcommand{\M}{\mathcal M}
\newcommand{\Y}{\mathcal Y}
\newcommand{\U}{\mathcal U}

\begin{document}

\title[Hecke operators on Siegel Eisenstein series]
{Hecke eigenvalues and relations for Siegel Eisenstein series of arbitrary degree, level, and character}

\author{Lynne H. Walling}
\address{School of Mathematics, University of Bristol, University Walk, Clifton, Bristol BS8 1TW, United Kingdom;
phone +44 (0)117 331-5245, fax +44 (0)117 928-7978}
\email{l.walling@bristol.ac.uk}




\begin{abstract} We evaluate the action of Hecke operators 
on Siegel Eisenstein series of 
arbitrary degree, level and character.  For square-free level, we simultaneously diagonalize the 
space with respect to all the Hecke operators,
 computing the eigenvalues explicitly, and obtain a multiplicity-one result.
For arbitrary level, we simultaneously diagonalize the space with respect to the Hecke operators
attached to primes not dividing the level, again computing the eigenvalues explicitly.
\end{abstract}

\maketitle
\def\thefootnote{}
\def\thefootnote{\arabic{footnote}}

\section{Introduction} 

\smallskip


Automorphic forms appear in almost every area of modern number theory; Eisenstein series are fundamental examples
of automorphic forms. 
In the case of classical  elliptic modular forms (i.e. holomorphic automorphic forms
of integral weight),
Eisenstein series are well-understood:  For instance, 
the Fourier expansions of a ``natural" basis of Eisenstein series have long been known; as well, it has long been known
that the space
of Eisenstein series of weight $k$, level $\stufe$ and character $\chi$ has a basis of simultaneous eigenforms for the
Hecke operators $\{T(p):\ p \text{ prime, } p\nmid\stufe\ \}$, and for $\{T(p):\ p \text{ prime } \}$ when $\stufe$ is square-free.
(see, e.g., chapter IV \cite{Ogg}).
The Fourier coefficients of these simultaneous eigenforms are (after appropriate normalization) the Hecke eigenvalues, and are
doubly-twisted divisor functions; that is, the $m$th Fourier coefficient of such a (normalized) form
of weight $k$ is
$$\sum_{d|m}\chi_1(d)\chi_2(m/d) d^{k-1}$$
where $\chi_1,\chi_2$ are Dirichlet characters,
reflecting the fact that the Fourier coefficients of Hecke eigenforms carry number theoretic information.

In the case of Siegel Eisenstein series, our knowledge is much less complete (for instance, we have limited knowledge of Fourier coefficients for arbitrary degree, level, and character).  However we do have analogues of some of the classical results regarding the action of Hecke operators.
By studying the abstract Hecke algebra, Evdokimov
(\cite{E1}, \cite{E2}) and Freitag (\cite {F})  showed that  the space of Siegel modular forms of arbitrary level and character can be diagonalized with respect to the Hecke operators associated to primes not dividing the level.  These  results also show that the subspace of Siegel Eisenstein series is invariant under these Hecke operators.
Further, in \cite{F} Freitag computed some of the eigenvalues of Siegel Eisenstein series under the Hecke operator $T(p)^m$ where $p$ is a prime not dividing the level, and $m$ is a suitable power.
Following his proof of the injectivity of the Hecke operator $T(p)$ when $p$ is a prime exactly dividing the level of a space of Siegel modular forms (\cite{B2}), in \cite{B3}, B\"ocherer applied
powers of $T(p)$ to the level 1 Siegel Eisenstein series, obtaining a basis for the space of Siegel Eisenstein
series of level $p$ and trivial character, and thereby also obtaining Fourier expansions for this basis.
In \cite{Wal}, for $p$ any prime, we applied an explicit set of matrices for $T(p), T_1(p^2), T_2(p^2)$ directly to a basis for the subspace of Siegel Eisenstein series of degree 2, square-free level, and arbitrary character; we then constructed a basis of simultaneous eigenforms and computed all their eigenvalues.
Recently in \cite{Klo}, Klosin used adelic methods to compute the
Hecke eigenvalues (for primes not dividing the level) on the space of hermitian forms on $U(2,2)$.

In the current paper, we extend the techniques of \cite{Wal} to allow arbitrary degree $n$, level $\stufe$, and character $\chi$ modulo $\stufe$.  In \S3, for each $\gamma\in Sp_n(\Z)$, we define a  Eisenstein series with character $\chi$ corresponding to the $\Gamma_0(\stufe)$-orbit of $\Gamma_{\infty}\gamma$.  We identify necessary conditions for one of these series to be nonzero, and in the case that $\stufe$ is square-free, we show that these conditions are also sufficient (Proposition 3.6).  Next we consider square-free level $\stufe$ and arbitrary character $\chi$ modulo $\stufe$.  
We subscript each element of our basis for this space of  Eisenstein series by some $\sigma=(\stufe_0,\ldots,\stufe_n)$ where $\stufe_0\cdots\stufe_n=\stufe$.
Using an explicit set of matrices giving the action of $T(q)$ where $q$ is a prime dividing $\stufe$, we directly evaluate the action of $T(q)$ on each basis element $\E_{\sigma}$, computing precisely the coefficients in the linear combination of 
 Eisenstein series that is equal to $\E_{\sigma}|T(q)$ (Theorem 4.1).  This allows us to show that we can (algorithmically) diagonalize the space of 
 Eisenstein series with respect to $\{T(q):\ q\text{ prime, }q|\stufe\ \}$, 
obtaining a new basis $\{\widetilde\E_{\sigma}\}_{\sigma}$ for the space. With $\sigma=(\stufe_0,\ldots,\stufe_n)$ and $q$ a prime dividing $\stufe_d$, we show that
$$\widetilde\E_{\sigma}|T(q)=\lambda_{\sigma}(q)\widetilde\E_{\sigma}
\text{ with }|\lambda_{\sigma}(q)|=q^{kd-d(d+1)/2}$$
 (Corollary 4.3; note that this recovers a result from \cite{B3} in the case that $\stufe$ is prime and $\chi$ is trivial).  Since we must have $k>n+1$ for absolute convergence of the Eisenstein series, this shows we have ``multiplicity-one"; that is, for $\widetilde\E_{\sigma}\not=\widetilde\E_{\rho}$, there is some prime $q|\stufe$ so that $\lambda_{\sigma}(q)\not=\lambda_{\rho}(q).$
Following this, for $1\le j\le n$, we directly evaluate the action of $T_j(q^2)$ on the original basis elements $\E_{\sigma}$ where we still assume that $q$ is a prime dividing $\stufe$ (Theorem 4.4).  We then compute the the $T_j(q^2)$-eigenvalues for each of the elements in the diagonalized basis (Corollary 4.5).

In \S 5, we consider Eisenstein series of arbitrary level $\stufe$ and arbitrary character $\chi$, and we directly evaluate
the action of $T(p), T_j(p^2)$ for primes
$p\nmid \stufe$ so that we can explicitly construct a basis of simultaneous eigenforms for these Hecke operators.  
To help us diagonalize the space with respect to these operators, we introduce
a group action of $\U_{\stufe}\times\U_{\stufe}$ on the space of Eisenstein series where
$\U_{\stufe}=(\Z/\stufe\Z)^{\times}$ (Proposition 5.1).  Then we use characters $\psi$ on this group to average Eisenstein series relative to this group action;
by orthogonality of characters, this yields a basis $\{\E_{\sigma,\psi}\}$ for the space of Eisenstein series,
where $\sigma$ indexes our natural basis.  In Corollary 5.3 we show that for any prime $p\nmid\stufe$,
$\E_{\sigma,\psi}|T(p)=\lambda_{\sigma,\psi}(p)\E_{\sigma,\psi}$ where
$$\lambda_{\sigma,\psi}(p)=\psi_1(p)\overline\psi_2(p^n)\prod_{i=1}^n(\psi_1\chi(p)p^{k-i}+1);$$
here $\psi(v,w)=\psi_1(v)\psi_2(w).$
In Theorem 5.4 we evaluate the action of $T_j(p^2)$ on the natural basis.  Theorems 5.2 and 5.4 show that the Hecke
operators commute with the group action of $\U_{\stufe}\times\U_{\stufe}$ on the space of Eisenstein series;
we let $R(w)$ be the operator corresponding to the action of the group element $(1,w)$.
Then to obtain more attractive eigenvalues, we introduce operators $T'_j(p^2)$ so that the algebra generated by
$$\{T(p),T'_j(p^2),R(p):\ \text{prime }p\nmid\stufe,\ 1\le j\le n\ \}$$ is the algebra generated by
$\{T(p),T_j(p^2),R(p):\ \text{prime }p\nmid\stufe,\ 1\le j\le n\ \},$
and in Corollary 5.5, we show $\E_{\sigma,\psi}|T'_j(p^2)=\lambda'_{j;\sigma,\psi}(p^2)\E_{\sigma,\psi}$ where
$$\lambda'_{j;\sigma,\psi}(p^2)=\bbeta_p(n,j)p^{(k-n)j+j(j-1)/2}\chi(p^j)\prod_{i=1}^j(\psi_2\chi(p)p^{k-i}+1)$$
(here $\bbeta_p(n,j)$ is the number of $j$-dimensional subspaces of an $n$-dimensional space over $\Z/p\Z$).
When $\stufe$ is square-free, we show $\E_{\sigma,\psi}=0$ unless $\psi_1=\prod_{0<d\le n}\overline\chi_{\stufe_d}^d$
and $\psi_2=\overline\chi_{\stufe_n}^2$ (where $\sigma=(\stufe_0,\ldots,\stufe_n)$), and then with such $\psi$,
$\widetilde\E_{\sigma}|T(p)=\lambda_{\sigma,\psi}(p)\widetilde\E_{\sigma}$ and 
$\widetilde\E_{\sigma}|T'_j(p^2)=\lambda'_{j;\sigma,\psi}(p^2)\widetilde\E_{\sigma}$ (where $\widetilde\E_{\sigma}$ is as in Corollary 4.3).

Note that when $\chi^2=1$, $T'_j(p^2)$ is the operator introduced in \cite{thetaI} and again in \cite{thetaII} so that
$\theta^{(n)}(\gen L)|T'_j(p^2)=\lambda'_j(p^2)\theta^{(n)}(\gen L)$ where $\theta^{(n)}(\gen L)$  is the averaged (``genus")
 theta series attached to the genus of the lattice $L$, which is equipped with a positive definite quadratic form.

As all the arguments herein are valid when considering non-holomorphic Eisenstein series in the variables $\tau$ and $s$
(defined in \S3), the results
extend immediately to incude these forms (with $k$ replaced by $k+s$ in the formulas).

\bigskip
\section{Notation and Hecke operators}
\smallskip

For $n\in\Z_+$, $Sp_n(\Z)$ denotes the group of $2n\times 2n$ integral, symplectic matrices; we often write
these in block form $\begin{pmatrix} A&B\\C&D\end{pmatrix}$ where $A,B,C,D$ are $n\times n$ matrices.  Subgroups of importance to us include
$$\Gamma_{\infty}=\left\{\begin{pmatrix} A&B\\0&D\end{pmatrix}\in Sp_n(\Z)\right\},$$
$$\Gamma_{\infty}^+=\left\{\begin{pmatrix} A&B\\0&D\end{pmatrix}\in Sp_n(\Z):\ \det A=1\ \right\},$$
$$\Gamma(\stufe)=\{\gamma\in Sp_n(\Z):\ \gamma\equiv I\ (\stufe)\ \},$$
$$\Gamma_0(\stufe)=\left\{\begin{pmatrix} A&B\\C&D\end{pmatrix}\in Sp_n(\Z): \ C\equiv 0\ (\stufe)\ \right\};$$
here $\stufe\in\Z_+$.
It is well-known that for $\gamma=\begin{pmatrix} *&*\\M&N\end{pmatrix},\gamma'=\begin{pmatrix} *&*\\M'&N'\end{pmatrix}\in Sp_n(\Z)$, we have
$\gamma'\in\Gamma_{\infty}^+\gamma$ if and only if $(M'\ N')\in SL_n(\Z)(M\ N)$ .
Suppose $\begin{pmatrix} K&L\\M&N\end{pmatrix}\in Sp_n(\Z)$; then $(M\ N)$ is a coprime symmetric pair, meaning that
$M,N$ are integral,
$M\,^tN$ is symmetric, and for every prime $p$, $\rank_p(M\ N)=n$, where $\rank_p$ denotes the rank over $\Z/p\Z$.
On the other hand, given any coprime symmetric pair of $n\times n$ matrices $(M\ N)$, there exists some
$\begin{pmatrix} K&L\\M&N\end{pmatrix}\in Sp_n(\Z)$.  We often write $(M,N)=1$ to denote that a pair of integral matrices $(M\ N)$ is coprime.

Degree $n>1$ Siegel modular forms have as their domain
$$\h_{n}=\{X+iY:\ X,Y\in\R^{n,n}_{\sym},\ Y>0\ \}$$
where $\R^{n,n}_{\sym}$ denotes the set of symmetric $n\times n$ matrices over $\R$,
and $Y>0$ means that the quadratic form represented by $Y$ is positive definite.
For $n,k,\stufe\in\Z_+$ and $\chi$ a Dirichlet character modulo $\stufe$, a
 Siegel modular form of degree $n$, weight $k$, level $\stufe$, character $\chi$ is a holomorphic function
$f:\h_n\to\C$ 
(holomorphic in all variables of $\tau\in\h_n$)
so that for all $\begin{pmatrix} A&B\\C&D\end{pmatrix}\in\Gamma_0(\stufe)$, we have
$$f((A\tau+B)(C\tau+D)^{-1})=\chi(\det D) \det(C\tau+D)^k f(\tau).$$
(Note that this generalises the definition of a classical modular form, except in that case, where $n=1$, we also require
$$\lim_{\tau\to i\infty}(c\tau+d)^{-k}f(\tau)<\infty$$
for all $\begin{pmatrix}a&b\\ c&d\end{pmatrix}\in SL_2(\Z).$)
We use $\M_k^{(n)}(\stufe,\chi)$ to denote the space of all such forms.

To define the Hecke operators, fix a prime $p$.
Set $\Gamma=\Gamma_0(\stufe)$ and take $f\in\M_k^{(n)}(\stufe,\chi)$.
We define
$$f|T(p)=p^{n(k-n-1)/2}\sum_{\gamma}\overline\chi(\gamma)\,f|\delta^{-1}\gamma$$
where $\delta=\begin{pmatrix} pI_n\\&I_n\end{pmatrix}$, $\gamma$ varies over
$$(\delta\Gamma\delta^{-1}\cap\Gamma)\backslash\Gamma,$$
and for $\gamma'=\begin{pmatrix} A&B\\C&D\end{pmatrix},$
$$f(\tau)|\gamma'=(\det \gamma')^{k/2}\,\det(C\tau+D)^{-k}\,f((A\tau+B)(C\tau+D)^{-1}).$$
We define
$$f|T_j(p^2)=p^{j(k-n-1)}\sum_{\gamma}\overline\chi(\gamma)\,f|\delta_j^{-1}\gamma$$
where $\delta_j=\begin{pmatrix} X_j\\&X_j^{-1}\end{pmatrix}$, $X_j=X_j(p)=\begin{pmatrix} pI_j\\&I_{n-j}\end{pmatrix}$,
and $\gamma$ varies over $$(\delta_j\Gamma\delta_j^{-1}\cap\Gamma)\backslash\Gamma.$$

To help us describe a set of matrices giving the action of each Hecke operator, we fix the following notation.
For $r,s\in\Z_{\ge 0}$ so that $r+s\le n$, let
$$X_{r,s}=X_{r,s}(p)=\begin{pmatrix} pI_r\\&I\\&&\frac{1}{p}I_s\end{pmatrix}\ (n\times n),$$
$$\K_{r,s}=\K_{r,s}(p)=X_{r,s}SL_n(\Z)X_{r,s}^{-1}\cap SL_n(\Z);$$
set $X_r=X_{r,0}$, $\K_{r}=\K_{r,0}.$

\begin{prop}\label{proposition 2.1}  Let $p$ be a prime, $f\in\M_k^{(n)}(\stufe,\chi).$
\item{(a)}  
We have
$$
f|T(p)
=p^{n(k-n-1)/2}\sum_{0\le r\le n}\chi(p^{n-r})\sum_{G,Y}
f|\begin{pmatrix} X_{r}^{-1}\\&\frac{1}{p}X_{r}\end{pmatrix}\begin{pmatrix} G^{-1}&Y\,^tG\\&^tG\end{pmatrix}
$$
where, for each $r$, $G$ varies over $SL_n(\Z)/\K_{r}(p)$ and $Y$ varies over
$$\Y_{r}(p)=\left\{\begin{pmatrix} Y_0\\&0\end{pmatrix}\in\Z_{\sym}^{n,n}:\ Y_0 \ r\times r,
\text{ varying modulo }p \right\}.$$
(Here $\Z^{n,n}_{\sym}$ denotes the set of integral, symmetric $n\times n$ matrices.)
\item{(b)}  For $1\le j\le n$,
\begin{align*}
&f|T_j(p^2)\\
&\quad=p^{j(k-n-1)}\sum_{n_0+n_2\le j}\chi(p^{j-n_0+n_2})\sum_{G,Y}
f|\begin{pmatrix} X_{n_0,n_2}^{-1}\\&X_{n_0,n_2}\end{pmatrix}\begin{pmatrix} G^{-1}&Y\,^tG\\&^tG\end{pmatrix}.
\end{align*}
Here, for each pair $n_0,n_2$,
$G=G_1G_2$, where $G_1$ varies over $SL_n(\Z)/\K_{n_0,n_2}(p)$,
$$G_2=\begin{pmatrix} I_{n_0}\\&G'\\&&I_{n_2}\end{pmatrix}$$ with $G'$ varying over
$SL_{n'}(\Z)/\,^t\K'_{j'}(p)$ where $n'=n-n_0-n_2$, $j'=j-n_0-n_2$,
$$\K'_{j'}=\begin{pmatrix} pI_{j'}\\&I\end{pmatrix} SL_{n'}(\Z)\begin{pmatrix}\frac{1}{p}I_{j'}\\&I\end{pmatrix}\cap SL_{n'}(\Z),$$
 and $Y$ varies over $\Y_{n_0,n_2}(p^2)$, the set
of all integral, symmetric $n\times n$ matrices 
$$\begin{pmatrix} Y_0&Y_2&Y_3&0\\^tY_2&Y_1/p&0\\^tY_3&0\\0\end{pmatrix}$$
with $Y_0$ $n_0\times n_0$, varying modulo $p^2$, $Y_1$ $j'\times j'$, varying
modulo $p$ provided $p\nmid\det Y_1$, 
and $Y_2,Y_3$  varying modulo $p$ with  $Y_3$ $n_0\times n_2$.
\end{prop}

\begin{proof}  Fix $\Lambda=\Z x_1\oplus\cdots\oplus\Z x_n$ (a reference lattice).  

By Lemma 6.2, as $G$ varies over $SL_n(\Z)/\K_{r}$, 
$\Omega=\Lambda GX_{r}$ varies over all lattices $\Omega$,
$p\Lambda\subseteq\Omega\subseteq\Lambda$ with $[\Lambda:\Omega]=p^{r}$.
Thus by Proposition 3.1 \cite{HW} and (the proof of) Theorem 6.1 in \cite{HW}, claim (1) of the proposition follows.

For $\Omega$ another lattice on $\Q\Lambda$, let
$\mult_{\{\Lambda:\Omega\}}(x)$ be the multiplicity of the value of $x$ among the invariant factors 
$\{\Lambda:\Omega\}$.  By Lemma 6.3, as $G_1$ varies over $SL_n(\Z)/\K_{n_0,n_2}(p)$,
$\Omega=\Lambda G_1X_{n_0,n_2}$ varies over all lattices $\Omega$, 
$p\Lambda\subseteq\Omega\subseteq\frac{1}{p}\Lambda$, with 
$\mult_{\{\Lambda:\Omega\}}(1/p)=n_2$,
$\mult_{\{\Lambda:\Omega\}}(p)=n_0$.
Then with $\Omega=\Omega_0\oplus\Omega_1\oplus\Omega_2$,
$\Lambda=\frac{1}{p}\Omega\oplus\Omega_1\oplus p\Omega_2$, as $G'$ varies over
$SL_{n'}(\Z)/\,^t\K'_{j'}(p)$ 
$$\Omega_1 G'\begin{pmatrix} I_{j'}\\&0\end{pmatrix} \text{ modulo }p$$
varies over all dimension $j'$ subspaces of $\Omega_1/p\Omega_1$.
Thus by Proposition 2.1 \cite{HW} and (the proofs of) Theorems 4.1 and 6.1 in \cite{HW}, claim (2) of the proposition follows.
\end{proof}

\noindent{\bf Remark.}
For $\stufe'\in\Z_+$ so that $p\nmid\stufe'$, by Lemma 6.1 we can choose $G$ in the above proposition so that
$G\equiv I\ (\stufe')$, and since
$\stufe'Y$ will vary over a set of representatives for $\Y_r(p)$ or $\Y_{n_0,n_2}(p^2)$ as $Y$ does,
we can choose $Y$ in the above proposition so that $Y\equiv 0\ (\stufe').$
Also, when $p|\stufe$, we have
$$f|T(p)=p^{n(k-n-1)/2}\sum_Y f|\begin{pmatrix} \frac{1}{p}I_n&\frac{1}{p}Y\\&I_n\end{pmatrix}$$
where $Y$ varies over $\Y_n(p)$, and
$$f|T_j(p^2)=p^{j(k-n-1)}\sum_{G,Y}f|\begin{pmatrix} X_j^{-1}\\&X_j\end{pmatrix}
\begin{pmatrix} G^{-1}&Y\,^tG\\&^tG\end{pmatrix}$$
where $G$ varies over $SL_n(\Z)/\K_{j}(p)$ and $Y$ varies over $\Y_{j,0}(p^2)$.
\smallskip

To describe the Hecke eigenvalues, we make use of the following elementary functions:  Fix $m\ge 0$.  With $r> 0$,
$$\bdelta(m,r)=\bdelta_p(m,r)=\prod_{i=0}^{r-1}(p^{m-i}+1),$$
$$\bmu(m,r)=\bmu_p(m,r)=\prod_{i=0}^{r-1}(p^{m-i}-1),$$
$$\bbeta(m,r)=\bbeta_p(m,r)=\bmu(m,r)/\bmu(r,r)$$
(note that $\bbeta_p(m,r)$ is the number of $r$-dimensional subspaces of an $m$-dimensional space over $\Z/p\Z$).
Take $\bdelta(m,0)=\bmu(m,0)=1.$
For $r<0$, we take $\bbeta(m,r)=0$.
As well, we will use the following functions:
With $p$ prime, $t\in\Z_+$, and $\F=\Z/p\Z$, let $\sym_p(t)$ be the number of invertible matrices in $\F^{t,t}_{\sym}$,
the set of symmetric $t\times t$ matrices over $\F$.  
More generally, let $\chi$ be a character of square-free
modulus $\stufe$, with $p|\stufe$; set
$$\sym_p^{\chi}(t)=\sum_{U\in\F^{t,t}_{\sym}}\chi_p(\det U),$$
and $$\sym_p^{\chi}(t-s,s)=\sum_U\chi_p\left(\det\begin{pmatrix} U_1&U_2\\^tU_2&0\end{pmatrix}\right)$$
where $U=\begin{pmatrix} U_1&U_2\\^tU_2&0\end{pmatrix}\in \F^{t,t}_{\sym}$ with $U_1$ of size $(t-s)\times(t-s)$
(so $\sym_p^{\chi}(t,0)=\sym_p^{\chi}(t)$).  Note that as $U$ varies over invertible matrices in $\F^{t,t}_{\sym}$,
so does $\overline U$ (where $U\overline U=I$ in $\F^{t,t}$), we have $\sym_p^{\overline\chi}(t)=\sym_p^{\chi}(t)$;
similarly, $\sym_p^{\overline\chi}(t-s,s)=\sym_p^{\chi}(t-s,s)$.
Also, take $\sym_p^{\chi}(0)=\sum_p^{\chi}(0,0)=1.$
Although we will not use the precise values of these functions
in this work, one can use the theory of quadratic forms over finitie fields to show that
for $p$ odd and $\varepsilon=\left(\frac{-1}{p}\right)$,
$$\sym_p^{\chi}(b,c)=\begin{cases}
\frac{p^{m^2+m-c}\bmu(b,b)}{\bmu\bdelta(m-c,m-c)}&\text{if $b+c=2m$ and $\chi_p=1$,}\\
\frac{\varepsilon^m p^{m^2}\bmu(b,b)}{\bmu\bdelta(m-c,m-c)}&\text{if $b+c=2m$, $\chi_p^2=1$, and $\chi_p\not=1$,}\\
\frac{p^{m^2+m}\bmu(b,b)}{\bmu\bdelta(m-c,m-c)}&\text{if $b+c=2m+1$ and $\chi_p=1$,}\\
0&\text{otherwise,}
\end{cases}$$
and for $p=2$,
$$\sym_2^{\chi}(b,c)=
\begin{cases} \frac{2^{m(m+1)}\bmu(b,b)}{\bmu\bdelta(m-c,m-c)}&\text{if $b+c=2m+1$,}\\
\frac{2^{m(m+1)}\bmu(b,b)}{\bmu\bdelta(m-1,m-1)}
\left(\frac{\bmu(2m-1,2c)}{\bmu\bdelta(m-1,c)}+\bmu\bdelta(m,c)\right)&\text{if $b+c=2m$.}
\end{cases}$$
(Here $\bmu=\bmu_p$, $\bdelta=\bdelta_p$.)

For  $p$ prime, $M\in \Z^{n,m}$, we write $\rank_pM$ to denote the rank of $M$ over $\Z/p\Z$;
we will also refer to this rank as the $p$-rank of $M$.

Recall that for
$\chi$ a character modulo $\stufe$ with $\stufe=\stufe'\stufe''$ so that $(\stufe',\stufe'')=1$, we know that $\chi$ factors
uniquely as $\chi_{\stufe'}\chi_{\stufe''}$
where $\chi_{\stufe'}$ is a character modulo $\stufe'$ and $\chi_{\stufe''}$ is a character modulo $\stufe''$.

In what follows, we will sometimes use the matrices
$G_{\pm}=\begin{pmatrix} -1\\&I_{n-1}\end{pmatrix}$ and $\gamma_{\pm}=\begin{pmatrix} G_{\pm}\\&G_{\pm}\end{pmatrix}$

\bigskip
\section {Defining Siegel Eisenstein series}
\smallskip

Fix $k,n,\stufe\in\Z_+$, $\chi$ a character modulo $\stufe$.  
To define Eisenstein series for $\Gamma_0(\stufe)$ with $k$ even, one can begin by defining
a $\Gamma(\stufe)$-Eisenstein series
$$\sum_{\delta^*}1(\tau)|\delta^* \text{ where }1(\tau)|\begin{pmatrix} A&B\\C&D\end{pmatrix}=\det(C\tau+D)^{-k}$$
and $\delta^*$ varies so that $\Gamma_{\infty}\Gamma(\stufe)=\cup_{\delta^*}\Gamma_{\infty}\delta^*$
(disjoint); then for $\gamma\in Sp_n(\Z)$, one can consider
$$\sum_{\delta^*,\delta}\overline\chi(\delta)\,1(\tau)|\delta^*\gamma\delta$$
where $\delta$ varies so that $\Gamma_{\infty}\gamma\Gamma_0(\stufe)=\cup_{\delta}\Gamma_{\infty}\Gamma(\stufe)\gamma\delta$
(disjoint).  
However, when $k$ is odd, these sums are not well-defined, since with 
$\gamma_{\pm}$ as defined in \S 2,
 we have $\gamma_{\pm}\in\Gamma_{\infty}$
and $1(\tau)|\gamma_{\pm}\delta^*=(-1)^k 1(\tau)|\delta^*$ for any $\delta^*\in\Gamma(\stufe)$.
Further, for $k$ even or odd, the latter sum is not well-defined unless $\chi$ is trivial on any matrix in $\Gamma_0(\stufe)$
that stablises $\Gamma_{\infty}\Gamma(\stufe)\gamma.$
Thus we proceed as follows.

Let $\delta^*\in\Gamma(\stufe)$ vary so that
$$\Gamma_{\infty}^+ \Gamma(\stufe)=\cup_{\delta^*}\Gamma_{\infty}^+\delta^*\ \text{(disjoint)},$$
and set $$\E^*(\tau)=\sum_{\delta^*} 1(\tau)|\delta^*.$$
Since $1(\tau)|\delta\delta^*=1(\tau)|\delta^*$ for $\delta\in\Gamma_{\infty}^+$,
$\E^*$ is well-defined. 
Further, provided $k>n+1$, $\E^*(\tau)$ converges absolutely uniformly on subsets
$\{\tau\in\h_n:\ \Im\tau\ge Y\ \}$ for any $Y\in\R^{n,n}_{\sym}$ with $Y>0$, and so $\E^*$ is analytic
(in all variables of $\tau$).
So suppose $k>n+1$.
Now take $\beta\in\Gamma_0(\stufe)$ so that
$$\Gamma_0(\stufe)=\cup_{\beta}\Gamma_{\infty} \Gamma(\stufe)\beta\ \text{(disjoint)},$$
and for $\gamma\in Sp_n(\Z)$, set
$$\E'_{\gamma}=\sum_{\beta} \overline\chi(\beta)\,\E^*|\gamma\beta\ +\ \sum_{\beta}\overline\chi(\gamma_{\pm}\beta)\,\E^*|\gamma_{\pm}\gamma\beta$$
(where $\gamma_{\pm}$ is as defined in \S2).
Note that 
$$\Gamma_{\infty}\gamma\Gamma_0(\stufe)=\cup_{\beta}\left(\Gamma_{\infty}^+\Gamma(\stufe)\gamma\beta\cup
\Gamma_{\infty}^+\Gamma(\stufe)\gamma_{\pm}\gamma\beta\right).$$
Let 
$$\Gamma_{\gamma}^+=\{\delta'\in\Gamma_0(\stufe):\ \Gamma_{\infty}^+\Gamma(\stufe)\gamma\delta'=
\Gamma_{\infty}^+\Gamma(\stufe)\gamma\ \},$$
the subgroup of $\Gamma_0(\stufe)$ that stabilizes $\Gamma_{\infty}^+\Gamma(\stufe)\gamma.$
Thus with $\delta$ varying over $\Gamma^+_{\gamma}\backslash\Gamma_0(\stufe)$, $\delta'$
over $\Gamma(\stufe)\backslash\Gamma^+_{\gamma},$ and noting that
$\E^*|\gamma_{\pm}=(-1)^k\E^*,$ we find
$$\E'_{\gamma}=\left(1+\chi(-1)(-1)^k\right)\sum_{\delta',\delta}\overline\chi(\delta'\delta)\,\E^*|\gamma\delta'\delta.$$
Since $\delta'\in\Gamma^+_{\gamma}$, we have $\gamma\delta'\gamma^{-1}\in\Gamma^+_{\infty}\Gamma(\stufe)$,
so $\E^*|\gamma\delta'=\E^*|\gamma.$  Hence
$$\E'_{\gamma}=(1+\chi(-1)(-1)^k)\sum_{\delta'}\overline\chi(\delta')\sum_{\delta}\overline\chi(\delta)\E^*|\gamma\delta.$$
Thus $\E'_{\gamma}=0$ if $\chi(-1)\not=(-1)^k$, or if $\chi$ is not trivial on $\Gamma^+_{\gamma}.$
Also note that when $\stufe\le 2$, we have $\gamma_{\pm}\in\Gamma(\stufe)$ and hence
$\E^*=\E^*|\gamma_{\pm}=(-1)^k\E^*$; so $\E'_{\gamma}=0$ if $\stufe\le 2$ and $k$ is odd.

Suppose 
$\stufe>2$ or $k$ is even; then
\begin{align*}
\lim_{\tau\to i\infty I}\E^*(\tau)
&=\#\left\{\delta^*\in \Gamma_{\infty}^+\backslash \Gamma_{\infty}^+\Gamma(\stufe):\ \delta^*\in\Gamma_{\infty}\ \right\}\\
&={\begin{cases} 2&\text{if $\stufe\le 2$},\\ 1&\text{if $\stufe>2$.}\end{cases}}
\end{align*}

Set
$$\E_{\gamma}=\frac{1}{2[\Gamma_{\gamma}^+:\Gamma(\stufe)]}\E_{\gamma}'.$$
Suppose $\chi(-1)=(-1)^k$, $\chi$ is trivial on $\Gamma_{\gamma}^+$, and suppose still that either $\stufe>2$ or $k$ is even;
we show that $\E_{\gamma}\not=0$.
We have
$$\E_{\gamma}(\tau)=\sum_{\delta^*,\delta}\overline\chi(\delta) 1(\tau)|\delta^*\gamma\delta$$
where $\delta^*$ varies over $\Gamma_{\infty}^+\backslash \Gamma_{\infty}^+\Gamma(\stufe)$ and
$\delta$ varies over $\Gamma_{\gamma}^+\backslash\Gamma_0(\stufe).$
We have $\delta^*\gamma\delta\gamma^{-1}\in\Gamma_{\infty}$ only if 
$\Gamma_{\infty}\Gamma(\stufe)\gamma\delta=\Gamma_{\infty}\Gamma(\stufe)\gamma,$ 
and since $\Gamma_{\infty}=\Gamma_{\infty}^+\cup\gamma_{\pm}\Gamma_{\infty}^+,$ we have
$\delta^*\gamma\delta\gamma^{-1}\in\Gamma_{\infty}$ only if $\delta\in\Gamma_{\gamma}^+$ or 
$\delta\in\Gamma_{\gamma}^+\gamma^{-1}\gamma_{\pm}\gamma$.
If $\delta\in\Gamma_{\gamma}^+$ then $\E^*|\gamma\delta\gamma^{-1}=\E^*$
and by assumption $\chi(\delta)=1$.
If $\delta=\beta\gamma^{-1}\gamma_{\pm}\gamma^{-1}$ for some $\beta\in\Gamma_{\gamma}^+$,
then with our assumptions,
$$\overline \chi(\delta)\E^*|\gamma\delta\gamma^{-1}=\overline\chi(\gamma_{\pm})\E^*|\gamma_{\pm}=\E^*.$$
Thus 
$$\lim_{\tau\to i\infty I}\E_{\gamma}(\tau)|\gamma^{-1}=
\#\{\delta^*,\delta:\ \delta^*\gamma\delta\gamma^{-1}\in\Gamma_{\infty}\ \}$$
(where $\delta^*$ varies over $\Gamma_{\infty}^+\backslash\Gamma_{\infty}^+\Gamma(\stufe)$,
$\delta$ varies over $\Gamma_{\gamma}^+\backslash\Gamma_0(\stufe)$), and this number is at least 1.
Hence $\E_{\gamma}\not=0$.
Noting that $\E'_{\gamma_{\pm}\gamma}=(-1)^k\E'_{\gamma}$, as $\gamma_{\sigma}$
varies over a set of representatives for 
$\Gamma_{\infty}\backslash Sp_n(\Z)/\Gamma_0(\stufe)$, the nonzero $\E'_{\gamma_{\sigma}}$ are
linearly independent.

Thus we have the following.

\begin{prop}\label{proposition 3.1} For $\gamma\in Sp_n(\Z)$,  
$\E_{\gamma}$ be as defined above.
\item{(a)}  We have $\E_{\gamma}\not=0$ if and only if
(1) $\chi(-1)=(-1)^k$, (2) $\chi$ is trivial on $\Gamma_{\gamma}^+$, and (3) either $\stufe>2$ or $k$ is even.
\item{(b)}  When $\E_{\gamma}\not=0$, we have $\E_{\gamma}(\tau)=\sum_{\delta}\overline\chi(\delta)\,1(\tau)|\gamma\delta$
where $\delta\in\Gamma_0(\stufe)$ varies so that 
$\Gamma_{\infty}^+\gamma\Gamma_0(\stufe)=
\cup_{\delta}\Gamma_{\infty}^+\gamma\delta$ (disjoint);
equivalently, with $\gamma=\begin{pmatrix} A&B\\C&D\end{pmatrix}$,
$$\E_{\gamma}(\tau)=\sum_{(M\ N)} \overline\chi(M,N)\det(M\tau+N)^{-k}$$
where $(M\ N)$ are coprime symmetric pairs varying so that
$$SL_n(\Z)(C\ D)\Gamma_0(\stufe)=\cup_{(M\ N)}SL_n(\Z)(M\ N) \text{ (disjoint)},$$
and
$\chi(M,N)=\chi(\delta)$ for $\delta\in\Gamma_0(\stufe)$ so that $(M\ N)\in SL_n(Z)(C\ D)\delta$.
\item{(c)} With $\gamma_{\sigma}$ varying over a set of
representatives for $\Gamma_{\infty}\backslash Sp_n(\Z)/\Gamma_0(\stufe)$, the non-zero $\E_{\gamma_{\sigma}}$ form 
a basis for $\Eis_k^{(n)}(\stufe,\chi),$ the space of Eisenstein series of degree $n$, weight $k$, level $\stufe$, and character $\chi$.
\end{prop}

\smallskip\noindent
{\bf Remarks.}  
\begin{enumerate}
\item 
Having fixed representatives $\{\gamma_{\sigma}\}$ for $\Gamma_{\infty}\backslash Sp_n(\Z)/\Gamma_0(\stufe)$,
we consider $\{\E_{\gamma_{\sigma}}\}$ to be a ``natural" basis for $\Eis^{(n)}_k(\stufe,\chi).$
\item 
For $s\in\C$ with $k+\Re s>n+1$, we can define a non-holomorphic Eisenstein series by replacing
$\det(M\tau+N)^{-k}$ by $$\det(M\tau+N)^{-k}|\det(M\tau+N)|^{-s}.$$
Then all the arguments and results herein are trivially modified to extend to these non-holomorphic forms.

\end{enumerate}

\smallskip

The next three propositions describe some useful relations when working with Eisenstein series; then for $\stufe$
square-free, we describe a convenient set of representatives for $\Gamma_{\infty}\backslash Sp_n(\Z)/\Gamma_0(\stufe)$
and how to evaluate $\chi(M,N).$

\begin{prop}\label{proposition 3.2}  Suppose $\gamma,\gamma'\in Sp_n(\Z)$ and $\delta\in\Gamma_0(\stufe)$ so that
$\Gamma_{\infty}^+\Gamma(\stufe)\gamma'=\Gamma_{\infty}^+\Gamma(\stufe)\gamma\delta.$  Then
$\E_{\gamma'}=\chi(\delta)\E_{\gamma}.$
\end{prop}

\begin{proof}  We have 
$\E'_{\gamma}=2\sum_h\overline\chi(\beta_h) \E^*|\gamma\beta_h$ where
$\Gamma_0(\stufe)=\cup_h\Gamma(\stufe)\beta_h$ (disjoint).  Thus
$\Gamma_0(\stufe)=\delta\Gamma_0(\stufe)=\cup_h\Gamma(\stufe)\delta\beta_h$
(recall that $\Gamma(\stufe)$ is a normal subgroup of $Sp_n(\Z)$); since $[\Gamma_0(\stufe):\Gamma(\stufe)]<\infty,$
this last union must be disjoint.  Thus
$$
\E'_{\gamma}=2\sum_h\overline\chi(\delta\beta_h) \E^*|\gamma\delta\beta_h
=2\overline\chi(\delta)\sum_h\overline\chi(\beta_h) \E^*|\gamma'\beta_h
=\overline\chi(\delta)\E'_{\gamma'}.
$$
Since $\Gamma^+_{\gamma'}=\delta\Gamma^+_{\gamma}\delta^{-1}$, 
we have $[\Gamma_0(\stufe):\Gamma^+_{\gamma'}]=[\Gamma_0(\stufe):\Gamma^+_{\gamma}]$ and so
the proposition follows.
\end{proof}

\begin{prop}\label{proposition 3.3}  Fix $\stufe\in\Z$.
Suppose $(M\ I)$, $(M'\ N')$, $(M''\ I)$ are coprime symmetric pairs so that
$(M''\ I)\equiv(M'\ N')\ (\stufe)$ and $(M'\ N')\in SL_n(\Z)(M\ I)\Gamma_0(\stufe).$
Then $(M''\ I)\in(M'\ N')\Gamma(\stufe)$ and hence
$(M''\ I)\in SL_n(\Z)(M\ I)\Gamma_0(\stufe).$
\end{prop}

\begin{proof}  Since $(M',N')=1$ and $N'\equiv I\ (\stufe)$, we have $(\stufe M',N')=1$.
Thus there is some $\begin{pmatrix} K'&L'\\M'&N'\end{pmatrix}\in Sp_n(\Z)$ with $L'\equiv0\ (\stufe)$,
and hence $K'\equiv I\ (\stufe)$.  Set
$$\gamma=\begin{pmatrix} K'&L'\\M'&N'\end{pmatrix}^{-1} \begin{pmatrix} I&0\\M''&I\end{pmatrix};$$
thus $\gamma\in\Gamma(\stufe)$ and 
$(M''\ I)=(M'\ N')\gamma\in SL_n(\Z)(M\ I)\Gamma_0(\stufe)$.
\end{proof}

\begin{prop}\label{proposition 3.4} For $\gamma\in Sp_n(\Z)$, there exists
some $\gamma''=\begin{pmatrix} I&0\\M''&I\end{pmatrix}\in Sp_n(\Z)$ so that
$\gamma\in\Gamma^+_{\infty}\gamma''\Gamma_0(\stufe).$
Equivalently, for $(M\ N)$ a coprime symmetric pair, there is some symmetric $M''$
so that $$(M\ N)\in SL_n(\Z)(M''\ I)\Gamma_0(\stufe).$$
\end{prop}

\begin{proof}  
Given $\gamma=\begin{pmatrix} *&*\\M&N\end{pmatrix}, \gamma''=\begin{pmatrix} I&0\\M''&I\end{pmatrix}\in Sp_n(\Z)$, 
recall that we have
$\gamma\in\Gamma^+_{\infty}\gamma''\Gamma_0(\stufe)$ if and only if
$(M''\ I)\in SL_n(\Z)(M\ N)\Gamma_0(\stufe)$.  
By Proposition 3.3, it suffices to show there is some $(M'\ N')\in SL_n(\Z)(M\ N)\Gamma_0(\stufe)$
so that $N'\equiv I\ (\stufe)$; we proceed algorithmically.

Fix a prime $q$ dividing $\stufe$
and take $t$ so that $q^t\parallel\stufe$.  
Using Lemma 6.1, we can choose $E_0,G_0\in SL_n(\Z)$ so that $E_0,G_0\equiv I\ (\stufe/q^t)$
and $E_0N\,^tG_0^{-1}\equiv\begin{pmatrix} N_1&0\\0&0\end{pmatrix}\ (q^t)$ where
$N_1$ is $d\times d$ and invertible modulo $q$ (so $d=\rank_qN$).  We can adjust
$E_0,G_0$ so that $N_1\equiv\begin{pmatrix} a\\&I\end{pmatrix}\ (q^t)$, some $a$.  Similarly, we can choose
$\begin{pmatrix} u&v\\w&x\end{pmatrix}\in SL_2(\Z)$ so that 
$\begin{pmatrix} u&v\\w&x\end{pmatrix}\equiv I\ (\stufe/q^t)$, 
$\begin{pmatrix} u&v\\w&x\end{pmatrix}\equiv\begin{pmatrix} a&0\\0&\overline a\end{pmatrix}\ (q^t)$
(where $a\overline a\equiv 1\ (q^t)$).
Then
$$\gamma_0=\begin{pmatrix} u&&v\\&I_{n-1}\\w&&x\\&&&I_{n-1}\end{pmatrix}
\in\Gamma_0(\stufe)$$
and $E_0(M\ N)\begin{pmatrix} G_0\\&^tG_0^{-1}\end{pmatrix}\gamma_0\equiv
\left(\begin{pmatrix} M_1&M_2\\M_3&M_4\end{pmatrix}\ \begin{pmatrix} I_d\\&0\end{pmatrix}\right)\ (q^t)$
with $M_1$ $d\times d$.
By the symmetry of $M\,^tN$, 
$M_3\equiv0\ (q^t)$; since $(M,N)=1$,
$M_4$ is invertible modulo $q$.  
Thus using Lemma 6.1 we can find $E_1',G_1'\in SL_{n-d}(\Z)$ so that 
$E_1', G_1'\equiv I\ (\stufe/q^t),$
$$M'_4=E_1'M_4 G_1'\equiv\begin{pmatrix} I\\&a'\end{pmatrix}\ (q^t), \text{ some }a'.$$  
Take $E_1=\begin{pmatrix} I_d\\&E_1'\end{pmatrix}$,
$G_1=\begin{pmatrix} I_d\\&G_1'\end{pmatrix}$.
Using the Chinese Remainder Theorem, we can choose $W'$ so that $W'\equiv 0\ (\stufe/q^t)$ and
$W'\equiv \begin{pmatrix} I_{n-d-1}\\&\overline a'\end{pmatrix}\ (q^t)$ where $a'\overline a'\equiv 1\ (q^t)$; set
$W=\begin{pmatrix} 0_d\\&W'\end{pmatrix}.$
Then with
$$(C\ D)=E_1E_0(M\ N)\begin{pmatrix} G_0\\&^tG_0^{-1}\end{pmatrix} \gamma_0\begin{pmatrix} G_1\\&^tG_1^{-1}\end{pmatrix}
\begin{pmatrix} I&W\\0&I\end{pmatrix},$$
we have $(C\ D)\in SL_n(\Z)(M\ N)\Gamma_0(\stufe)$, $(C\ D)\equiv(M\ N)\ (\stufe/q^t),$ and
$D\equiv I\ (q^t).$

Next, suppose $p$ is another prime dividing $\stufe$ with $p^r\parallel\stufe$.  Applying the above process to the pair
$(C\ D)$, we obtain a pair $(C'\ D')\in SL_n(\Z)(M\ N)\Gamma_0(\stufe)$ with
$(C'\ D')\equiv (M\ N)\ (\stufe/(q^tp^r))$ and $D'\equiv I\ (q^tp^r)$.
Continuing, we obtain $(M'\  N')\in SL_n(\Z)(M\ N)\Gamma_0(\stufe)$ with $N'\equiv I\ (\stufe)$.
Applying Proposition 3.3 completes the proof.
\end{proof}

\begin{prop}\label{proposition 3.5}  
Let $(M\ N)$ be a coprime symmetric pair.  There is some symmetric matrix $M_{\sigma}$ so that $(M\ N)\in SL_n(\Z)(M_{\sigma}\ I)\Gamma_0(\stufe)$
and for each prime $q$ with $q\parallel\stufe$, we have $M_{\sigma}\equiv\begin{pmatrix}I_{d}\\&0\end{pmatrix}$ where $d=d(q)=\rank_qM.$
Thus when $\stufe$ is square-free and $M_{\sigma}$ is as above, we have 
$(M\ N)\in SL_n(\Z)(M_{\sigma}\ I)\Gamma_0(\stufe)$
if and only if $\rank_qM=\rank_qM_{\sigma}$ for all primes $q|\stufe$.
Further, with $\stufe$ square-free, we can take $M_{\sigma}$ diagonal, and we have
$$SL_n(\Z)(M_{\sigma}\ I)\Gamma_0(\stufe)=GL_n(\Z)(M_{\sigma}\ I)\Gamma_0(\stufe).$$
\end{prop}

\begin{proof}  
First note that if $(M\ N), (M_{\sigma}\ I)$ are coprime symmetric pairs with $(M\ N)\in SL_n(\Z)(M_{\sigma}\ I)\Gamma_0(\stufe)$,
then $\rank_qM=\rank_qM_{\sigma}$ for all primes $q|\stufe$, since elements of $\Gamma_0(\stufe)$ are of the form
$\begin{pmatrix}A&B\\C&D\end{pmatrix}$ with $C\equiv0\ (\stufe)$ and thus $A$ invertible modulo $\stufe$.

From Proposition 3.4, we know there is a symmetric matrix $M''$ so that $(M\ N)\in SL_n(\Z)(M''\ I)\Gamma_0(\stufe)$.  
Suppose $q$ is prime with $q\parallel\stufe$; let $d=d(q)=\rank_q M''$.  
If $d=0$ then set $E_q=I_n$ and $\gamma_q=I_{2n}$. Otherwise,
using Lemma 6.1, we can choose $E_q\in SL_n(\Z)$ so that $E_q\equiv I\ (\stufe/q)$
and $E_qM''\,^tE_q\equiv\begin{pmatrix} a\\&I_{d-1}\\&&0\end{pmatrix}\ (q);$
choose $\begin{pmatrix} w&x\\y&z\end{pmatrix}\in SL_n(\Z)$ so that
$$\begin{pmatrix} w&x\\y&z\end{pmatrix}\equiv I\ (\stufe/q),\ 
\begin{pmatrix} w&x\\y&z\end{pmatrix}\equiv \begin{pmatrix} \overline a&\overline a-1\\0&a\end{pmatrix}\ (q)$$ where $a\overline a\equiv 1\ (q)$, 
and set
$$\gamma_q=\begin{pmatrix} ^tE_q\\&E_q^{-1}\end{pmatrix}\begin{pmatrix} w&&x\\&I_{n-1}&&0\\y&&z\\&0&&I_{n-1}\end{pmatrix}.$$
Set $E=\prod_{q\parallel\stufe}E_q$, $\gamma=\prod_{q\parallel\stufe}\gamma_q.$
Thus $E\in SL_n(\Z)$, $\gamma\in\Gamma_0(\stufe)$; set $(M'\ N')=E(M''\ I)\gamma.$
So $(M'\ N')\equiv(M_{\sigma}\ I)\ (\stufe)$ for some symmetric $M_{\sigma}$ with
$M_{\sigma}\equiv\begin{pmatrix}I_{d(q)}\\&0\end{pmatrix}\ (q)$ for all primes $q\parallel\stufe$.
Then by Proposition 3.3, $(M\ N)\in SL_n(\Z)(M_{\sigma}\ I)\Gamma_0(\stufe).$

Suppose $\stufe$ is square-free; then 
we can use the Chinese Remainder Theorem to choose $M_{\sigma}$ diagonal with
$M_{\sigma}\equiv\begin{pmatrix} I_{d(q)}\\&0\end{pmatrix}\ (q)$ for each prime $q|\stufe$.
Also,
\begin{align*}
&SL_n(\Z)(M_{\sigma}\ I)\Gamma_0(\stufe)\\
&\quad = SL_n(\Z)(M_{\sigma}\ I)\Gamma_0(\stufe)\cup SL_n(\Z)(M_{\sigma}\ I)\gamma_{\pm}\Gamma_0(\stufe)\\
&\quad = SL_n(\Z)(M_{\sigma}\ I)\Gamma_0(\stufe)\cup SL_n(\Z)G_{\pm}(M_{\sigma}\ I)\Gamma_0(\stufe)\\
&\quad = GL_n(\Z)(M_{\sigma}\ I)\Gamma_0(\stufe).
\end{align*}
This proves the proposition.
\end{proof}

Using Proposition 3.5, we 
fix a set of representatives $\left\{\gamma_{\sigma}=\begin{pmatrix} I&0\\M_{\sigma}&I\end{pmatrix}\right\}_{\sigma}$ for
$\Gamma_{\infty}\backslash Sp_n(\Z)/\Gamma_0(\stufe)$ so that when $q$ is a prime with $q\parallel\stufe$,
we have $M_{\sigma}\equiv\begin{pmatrix} I_d\\&0\end{pmatrix}\ (q)$ for some $d=d(q)$,
and when $\stufe$ is square-free, $M_{\sigma}$ is diagonal.
Let $\E_{\sigma}$ denote $\E_{\gamma_{\sigma}}.$

\begin{prop}\label{proposition 3.6}  Suppose that $\chi(-1)=(-1)^k$, and either $\stufe>2$ or $k$ is even.
\item{(1)}  Suppose $\E_{\sigma}\not=0$ and $q$ is prime so that $q\parallel\stufe$; let $d=d(q)=\rank_qM_{\sigma}.$
If $0<d<n$ then $\chi_q^2=1$.
\item{(2)}  Suppose $\stufe$ is square-free.  Then $\E_{\sigma}\not=0$ if and only if
$\chi_q^2=1$ for all primes $q|\stufe$ so that $0<\rank_qM_{\sigma}<n$.
\end{prop}

\begin{proof}  
(1)  Suppose we have a prime $q\parallel\stufe$ with $0<d<n$ where $d=\rank_qM_{\sigma}.$
Choose $u\in\Z$ so that $q\nmid u$, and (using Lemma 6.1) choose $\begin{pmatrix} w&x\\y&z\end{pmatrix}\in SL_n(\Z)$
so that $\begin{pmatrix} w&x\\y&z\end{pmatrix}\equiv I\ (\stufe/q)$ and
$\begin{pmatrix} w&x\\y&z\end{pmatrix}\equiv \begin{pmatrix} u\\&\overline u\end{pmatrix}\ (q)$ where $u\overline u\equiv 1\ (q).$
Set 
$E=\begin{pmatrix} w&&x\\&I_{n-2}\\y&&z\end{pmatrix},$
$$\delta=\begin{pmatrix} z&&&x\\&I_{n-2}&&&0\\&&z&&&x\\y&&&w\\&0&&&I_{n-2}\\&&y&&&w\end{pmatrix}
\begin{pmatrix} 1&&1-u^2\\&I_{n-1}&&0\\&&1\\&&&I_{n-1}\end{pmatrix}.$$
Thus $E\in SL_n(\Z)$, $\delta\in\Gamma_0(\stufe)$, and $E(M_{\sigma}\ I)\delta\equiv(M_{\sigma}\ I)\ (\stufe).$
So $\delta\in\Gamma_{\gamma_{\sigma}}^+$, and thus
$\E_{\sigma}=\E_{\sigma}|\delta$.  We also have
$\E_{\sigma}|\delta=\chi_q(u^2)\E_{\sigma}.$  Since $\E_{\sigma}\not=0$, this means
$\chi_q^2(u)=1$, and this holds for all $u\in\Z$ where $q\nmid u$.  Hence $\chi_q^2=1$.

(2)  Now suppose $\stufe$ is square-free, and
that for each prime $q|\stufe$ with $0<\rank_qM_{\sigma}<n$, we have $\chi_q^2=1$.
To show $\E_{\sigma}\not=0$, we need to
show $\chi$ is trivial on $\Gamma_{\gamma_{\sigma}}^+$.
To do this, we show that for all primes $q|\stufe$, $\chi_q$ is trivial on $\Gamma_{\gamma_{\sigma}}^+$.
So take $\beta=\begin{pmatrix} A&B\\C&D\end{pmatrix}\in\Gamma^+_{\gamma_{\sigma}}$.  Thus there exist
$\delta=\begin{pmatrix} ^tE^{-1}&WE\\&E\end{pmatrix}\in\Gamma_{\infty}^+$,
$\beta'\in\Gamma(\stufe)$ so that $\delta\beta'\gamma_{\sigma}\beta=\gamma_{\sigma}$.
Thus $E(M_{\sigma}A\ M_{\sigma}B+D)\equiv (M_{\sigma}\ I)\ (\stufe).$
Fix a prime $q|\stufe$, and set $d=\rank_qM_{\sigma}$.  

When $d=0$, we have $ED\equiv I\ (q)$, so $\det D\equiv\det\overline E\equiv 1\ (q)$
and $\chi_q(\det D)=1$.  When $d=n$, we have
$EA\equiv I\equiv A\,^tD\ (q)$, so $\det D\equiv\det E\equiv 1\ (q)$ and $\chi_q(\det D)=1$.

Now suppose $0<d<n$.
Write 
$$A=\begin{pmatrix} A_1&A_2\\A_3&A_4\end{pmatrix},\ 
D=\begin{pmatrix} D_1&D_2\\D_3&D_4\end{pmatrix},\ 
E=\begin{pmatrix} E_1&E_2\\E_3&E_4\end{pmatrix}$$
where $A_1, D_1, E_1$ are $d\times d$.
Since $EM_{\sigma}A\equiv\begin{pmatrix} I_d\\&0\end{pmatrix}\ (q),$
we have $E_3(A_1\ A_2)\equiv 0\ (q)$.  Since $A$ is invertible modulo $q$, the rows of
$(A_1\ A_2)$ are linearly independent modulo $q$, and hence we must have $E_3\equiv0\ (q),$
$\rank_qE_1=d$, $\rank_qE_4=n-d$, and
$$1=\det E\equiv\det E_1\cdot\det E_4\ (q).$$
Also, since
$$E_1(A_1\ A_2)\equiv(I_d\ 0)\ (q),\ E_4(D_3\ D_4)\equiv(0\ I_{n-d})\ (q),$$
we have $A_2,D_3\equiv0\ (q)$, $A_1\equiv\overline E_1\ (q),$ $D_4\equiv\overline E_4\ (q).$
Since $A\,^tD\equiv I\ (q)$, we must have $D_1\equiv\,^tE_1\ (q).$  Thus we have 
$$\det D\equiv \det E_1\cdot\det \overline E_4
\equiv (\det E_1)^2\ (q)$$
and hence
$$\chi_q(\det D)=\chi_q^2(\det E_1)=1.$$
Thus with $\beta\in\Gamma_{\gamma_{\sigma}}^+$, for all primes $q|\stufe$ we have $\chi_q(\beta)=1$;
consequently, by Proposition 3.1, $\E_{\sigma}\not=0$.
\end{proof}

\begin{prop}\label{proposition 3.7}  Suppose 
$\E_{\sigma}\not=0$, $(M\ N)\in SL_n(\Z)(M_{\sigma}\ I)\gamma$ where
$\gamma\in\Gamma_0(\stufe)$, and fix a prime $q$ so that $q\parallel\stufe$.
There are $E_0,E_1\in SL_n(\Z)$ so that
$$E_0ME_1\equiv\begin{pmatrix} M_1&0\\0&0\end{pmatrix}\ (q)$$ 
with $M_1$ invertible modulo $q$; 
for any such $E_0,E_1$ we have
$$E_0N\,^tE_1^{-1}\equiv\begin{pmatrix} N_1&N_2\\0&N_4\end{pmatrix}\ (q)$$ 
and 
$\chi_q(\gamma)=\chi_q(\det\overline M_1\cdot\det N_4).$
Further, for any $G\in GL_n(\Z)$, we have
$$\chi_q(GM,GN)=\chi_q(\det G)\chi_q(M,N)=\chi_q(MG^{-1},N\,^tG).$$
\end{prop}

\begin{proof}  By assumption, $(M\ N)=E(M_{\sigma}\ I)\gamma$ for some $E\in SL_n(\Z)$.
Set $d=\rank_qM_{\sigma}$.  If $d=0$ then $N\equiv ED\ (q)$ so $\chi_q(\gamma)=\chi_q(\det N).$
If $d=n$ then $M\equiv EA\equiv E\,^tD^{-1}\ (q)$ so $\chi_q(\gamma)=\chi_q(\det\overline M)$
(where $M\overline M\equiv I\ (q)$).

Suppose $0<d<n$.  By Proposition 3.5, we know $\rank_qM=\rank_qM_{\sigma}=d$, so
there are $E_0,E_1\in SL_n(\Z)$ so that $E_0ME_1\equiv\begin{pmatrix} M_1&0\\0&0\end{pmatrix}\ (q)$
with $M_1$ $d\times d$ and invertible modulo $q$.  Then by the symmetry of $M\,^tN$, we have
$E_0N\,^tE_1^{-1}\equiv \begin{pmatrix} N_1&N_2\\0&N_4\end{pmatrix}\ (q)$ with $N_1$ $d\times d$, and
$N_4$ invertible modulo $q$ since $(M,N)=1$.  
Set $E_2=E_0E$; given the shape of $M_{\sigma}$ and of
$E_0ME_1$, we must have $E_2\equiv\begin{pmatrix} E'&*\\0&E''\end{pmatrix}\ (q)$ with $E'$ $d\times d$
and invertible modulo $q$.
Hence
$$E_0(M\ N)\begin{pmatrix} E_1\\&^tE_1^{-1}\end{pmatrix} =E_2(M_{\sigma}\ I)\begin{pmatrix} ^tE_2\\&E_2^{-1}\end{pmatrix} \gamma'
=(M'\ I)\gamma'$$
where $$\gamma'=\begin{pmatrix} ^tE_2^{-1}\\&E_2\end{pmatrix}\gamma\begin{pmatrix} E_1\\&^tE_1^{-1}\end{pmatrix}\in\Gamma_0(\stufe)
\text{ and } M'\equiv\begin{pmatrix} E'\,^tE'\\&0\end{pmatrix}\ (q).$$
Write $\gamma'=\begin{pmatrix} A&B\\C&D\end{pmatrix}$, $A=\begin{pmatrix} A_1&A_2\\A_3&A_4\end{pmatrix}$, $D=\begin{pmatrix} D_1&D_2\\D_3&D_4\end{pmatrix}$
where $A_1, D_1$ are $d\times d$.  
Since $\begin{pmatrix} M_1\\&0\end{pmatrix}\equiv M'A\ (q)$, we have
$A_2\equiv 0\ (q)$, $A_1$ invertible modulo $q$, and $M_1\equiv E'\,^tE'A_1\ (q).$  
Then since $A\,^tD\equiv I\ (q)$, we have $D_3\equiv 0\ (q)$,
$A_1\,^tD_1\equiv I\ (q)$, $N_4\equiv D_4\ (q).$  
Thus 
$$\chi_q(\gamma)=\chi_q(\gamma')=\chi_q^2(E')\chi_q(\det\overline M_1\cdot\det N_4)$$
(where $M_1\overline M_1\equiv I\ (q)$).
Since $0<d<n$ and $\E_{\sigma}\not=0$, we know from Proposition 3.6 that $\chi_q^2=1$.

Suppose $(M\ N)=E(M_{\sigma}\ I)\gamma$ where $E\in SL_n(\Z)$, $\gamma\in\Gamma_0(\stufe)$;
take $G\in GL_n(\Z)$.  If $\det G=1$ then the above argument shows $\chi_q(GM,GN)=\chi_q(M,N)$.
Say $\det G=-1$; then $E'=GEG_{\pm}\in SL_n(\Z)$ and
$$G(M\ N)\equiv E'G_{\pm}(M_{\sigma}\ I)\gamma\equiv E'(M_{\sigma}\ I)\gamma_{\pm}\gamma\ (q).$$
Hence $\chi_q(GM,GN)=\chi_q(\gamma_{\pm}\gamma)=\chi_q(-1)\chi_q(M,N).$
Somewhat similarly,
$$(MG^{-1}\ \,N\,^tG)=E(M_{\sigma}\ I)\gamma\begin{pmatrix} G^{-1}\\&^tG\end{pmatrix},$$
so $\chi_q(MG^{-1},N\,^tG)=\chi_q\left(\gamma\begin{pmatrix} G^{-1}\\&^tG\end{pmatrix}\right)=\chi_q(\gamma)\chi_q(\det G).$
\end{proof}

\bigskip
\section{Hecke operators on Siegel Eisenstein series of square-free level}
\smallskip

Throughout this section, we assume $\stufe$ is square-free, $\chi$ is a character modulo $\stufe$
so that $\chi(-1)=(-1)^k$;
further, we assume either $\stufe>2$ or $k$ is even.

Let $\sigma$ be a ``multiplicative partition" of $\stufe$, meaning
$\sigma=(\stufe_0,\ldots,\stufe_n)$ where $\stufe_i\in\Z_+$ and $\stufe_0\cdots\stufe_n=\stufe$;
take $M_{\sigma}$ to be a diagonal $n\times n$ matrix so that for each $d$, $0\le d\le n$, 
we have $M_{\sigma}\equiv\begin{pmatrix} I_d\\&0\end{pmatrix}\ (\stufe_d)$.
By Proposition 3.1, as we vary $\sigma$,
the matrices $\gamma_{\sigma}=\begin{pmatrix} I&0\\M_{\sigma}&I\end{pmatrix}$ give us a set of representatives for
$\Gamma_{\infty}\backslash Sp_n(\Z)/\Gamma_0(\stufe)$, and by Proposition 3.5 
we have $\Gamma_{\infty}\gamma_{\sigma}\Gamma_0(\stufe)=\Gamma_{\infty}^+\gamma_{\sigma}\Gamma_0(\stufe)$.
Thus given any coprime symmetric pair $(M\ N)$, there is a unique multiplicative partition $\sigma$ of $\stufe$ so that
$(M\ N)\in SL_n(\Z)(M_{\sigma}\ I)\Gamma_0(\stufe).$

To ease notation, 
we write $\E_{\sigma}$ to denote $\E_{\gamma_{\sigma}}.$

\begin{thm}\label{Theorem 4.1}  
Fix a prime $q|\stufe$ and a multiplicative partition $$\sigma'=(\stufe'_0,\ldots,\stufe'_n)$$ of
$\stufe/q$; let $X_d=X_d(q)$ (as defined in \S 2).
For $0\le d\le n$, let 
$\sigma_{d}=(\stufe_0,\ldots,\stufe_n)$ where
$$\stufe_i=\begin{cases} \stufe_i'&\text{if $i\not=d$,}\\ q\stufe'_{d}&\text{if $i=d$.}\end{cases}$$
Then when $\E_{\sigma_d}\not=0$, we have
\begin{align*}
\E_{\sigma_d}|T(q)
&=q^{kd-d(d+1)/2}
\overline\chi_{\stufe/q}\left(qX_d^{-1} M_{\sigma_{d}},X_d^{-1}\right)\\
&\quad \cdot \sum_{t=0}^{n-d} q^{-dt-t(t+1)/2}
\bbeta_q(d+t,t)\sym_q^{\chi}(t)\E_{\sigma_{d+t}}
\end{align*}
(with 
$\sym_q^{\chi}(t)$ as defined in \S2).
\end{thm}

\begin{proof}  To ease notation further, temporarily
write $\E_{d'}$ for $\E_{\sigma_{d'}}$ and $M_{d'}$ for $M_{\sigma_{d'}}$.  Also,
write $\K_d$ for $\K_d(q)$, $\Y_n$ for $\Y_n(q)$, $X_r$ for $X_r(q)$,
$\bbeta(m,r)$ for $\bbeta_q(m,r).$

By Proposition 2.1, we have
$$\E_{d}(\tau)|T(q)
=q^{-n(n+1)/2} \sum_{M,N,Y} \overline\chi(M,N)\,\det(M\tau/q+MY/q+N)^{-k}$$
where $SL_n(\Z)(M\ N)$ varies over $SL_n(\Z)(M_d\ I)\Gamma_0(\stufe)$ and
$Y$ varies over $\Y_n$; recall that we can take $Y\equiv 0\ (\stufe/q)$.
(Note that in Proposition 2.1, when $p|\stufe$ we have $\chi(p^{n-r})=0$ unless $r=n$.)
Using left multiplication from $SL_n(\Z)$ to adjust each representative $(M\ N)$, we can
assume $q$ divides the lower $n-d$ rows of $M$. 
Set $$(M'\ N')=X_d (M/q\ MY/q+N)
=\frac{1}{q} X_d (M\ MY+qN);$$
clearly $M', N'$ are integral given our assumption that $q$ divides the lower $n-d$ rows of $M$.
We know the upper $d$ rows of $M$ are linearly independent modulo $q$, as are the lower $n-d$ rows of $N$.
Thus $(M',N')=1$, and with $d'=\rank_qM'$, we have $d'\ge d$.  Since $\rank_{q'}M_{d'}=\rank_{q'}M_d$ for all primes $q'|\stufe/q$,
by Proposition 3.5 we have $(M'\ N')\in SL_n(\Z)(M_{d'}\ I)\Gamma_0(\stufe).$
Also, we have
$$\det(M\tau/q+MY/q+N)^{-k}=q^{kd}\det(M'\tau+N')^{-k}.$$

Reversing, given $(M'\ N')\in SL_n(\Z)(M_{d'}\ I)\Gamma_0(\stufe)$ (with $d'\ge d$), we need to identify the equivalence classes
$SL_n(\Z)(M\ N)\in SL_n(\Z)(M_d\ I)\Gamma_0(\stufe)$ and $Y\in\Y_n$ so that
$$\frac{1}{q}X_d(M\ \,MY+N)\in SL_n(\Z)(M'\ N').$$
Equivalently, we need to identify $Y\in \Y_n$ and the equivalence classes
$$SL_n(\Z)qX_d^{-1}E(M'\ (N'-M'Y)/q)\in SL_n(\Z)(M_d\ I)\Gamma_0(\stufe)$$
where $E\in SL_n(\Z)$ and $(M'\ N')$ is a coprime symmetric pair.  
For $E\in SL_n(\Z)$, we have $X_d^{-1}EX_d\in SL_n(\Z)$ if and only if $E\in\K_d$;
thus we need to identify $Y\in\Y_n$ and $E\in\K_d\backslash SL_n(\Z)$ so that
$$qX_d^{-1}E(M'\ (N'-M'Y)/q)$$
is an integral, coprime pair with $\rank_q qX_d^{-1}EM'=d$
(that $M\,^tN$ is symmetric is automatic).
For each coprime symmetric pair $(M'\ N')$,
let $\calC_d(M',N')$ be the set of all  pairs $(E,Y)$
that meet the above criteria (note that $\calC_d(M',N')$ could be empty); then
$$\E_d(\tau)|T(q)=q^{kd-n(n+1)/2}\sum_{(M',N')}c_d(M',N')\,\det(M'\tau+N')^{-k}$$
where $$c_d(M',N')=\sum_{E,Y}\overline\chi(qX_d^{-1}EM',X_d^{-1}E(N'-M'Y)),$$
with the sum over all $(E,Y)\in\calC_d(M',N')$.

We also know that $\Eis_k^{(n)}(\stufe,\chi)$ is equal to
$$\spn\{(C\tau+D)^{-k}:\ (C\ D) \text{ coprime, symmetric }\}\cap\M_k^{(n)}(\stufe,\chi),$$
and $\M_k^{(n)}(\stufe,\chi)$ is invariant under the Hecke operators.  Hence $\E_d|T(q)$ is
again an Eisenstein series, and so the above discussion shows that
$$\E_d|T(q)=q^{kd-n(n+1)/2}\sum_{d'\ge d}c_d(M_{d'},I)\E_{d'}.$$  Thus we need to compute
$c_d(M_{d'},I)$ for each $d'\ge d$.

Fix $d'\ge d$, and 
choose $E\in\K_d\backslash SL_n(\Z)$; note that we can choose $E\equiv I\ (\stufe/q).$
With $Y\in\Y_n$, set
$$(M\ N)=qX_d^{-1}E(M_{d'}\ (I-M_{d'}Y)/q).$$
To have $\rank_qM=d$, we need the top $d$ rows of $EM_{d'}$ to have $q$-rank $d$; by Lemma 6.4(a),
the number of such $E$ is $q^{d(n-d')}\bbeta(d',d).$  
Also, since $M_{d'}\equiv\begin{pmatrix} I_{d'}\\&0\end{pmatrix}\ (q)$,
the upper left $d\times d'$ block of $E$ must have $q$-rank $d$; thus using
left multiplication from $\K_d$, we can assume $E=\begin{pmatrix} E'&W\\0&I\end{pmatrix}$ where $E'\in SL_{d'}(\Z).$
(Note that we can still assume that $E\equiv I\ (\stufe/q)$.)
So fix such $E$ (and thus fix $M$).  Set $G=\begin{pmatrix}E'\\&I\end{pmatrix}$;
so $\begin{pmatrix} G^{-1}\\&^tG\end{pmatrix}\in\Gamma_0(\stufe)$.  We know $N$
is integral if and only if $EN\,^tG=X_d^{-1}(E\,^tG-EM_{d'}G^{-1}\cdot GY\,^tG)$ is integral; also, when $N$ is integral, $(M,N)=1$ if and only if
$(MG^{-1},N\,^tG)=1$.  
Write $E'\,^tE'=\begin{pmatrix} E_1&E_2\\^tE_2&E_3\end{pmatrix}$, $W=\begin{pmatrix} W_1\\W_2\end{pmatrix}$,
$$G Y\,^tG=\begin{pmatrix} Y_1&Y_2&Y_3\\^tY_2&Y_4&Y_5\\^tY_3&^tY_5&Y_6\end{pmatrix}$$
where $E_1,Y_1$ are $d\times d$ and symmetric, $E_3,Y_4$ are $(d'-d)\times (d'-d)$ (and symmetric), and $W_1$ is $d\times(n-d')$.
We have $EM_{d'}G^{-1}\equiv \begin{pmatrix} I_{d'}\\&0\end{pmatrix}\ (q)$, so $N\,^tG$ is integral if and only if $(Y_1\ Y_2\ Y_3)\equiv(E_1\ E_2\ W_1)\ (q).$
When $N\,^tG$ is integral, we have
$$N\,^tG\equiv\begin{pmatrix} (E_1-Y_1)/q&(E_2-Y_2)/q&(W_1-Y_3)/q\\ 0&E_3-Y_4&W_2-Y_5\\0&0&I\end{pmatrix}\ (q),$$
so $(MG^{-1},N\,^tG)=1$ if and only if $\rank_q(E_3-Y_4)=d'-d.$
As $Y_4$ varies over symmetric $(d'-d)\times(d'-d)$ matrices modulo $q$,
so does $E_3-Y_4$.  
Recall that we can choose $Y\equiv0\ (\stufe/q)$; thus for $E,Y$ as above, we have $M_{d'}\equiv M_d\ (\stufe/q)$, so
$$\chi_{\stufe/q}(M,N)
= \chi_{\stufe/q}(qX_d^{-1}M_d,X_d^{-1})=\overline\chi_{\stufe/q}(\overline q X_dM_d,X_d)$$
and $$\chi_q(M,N)=\chi_q(MG^{-1},N\,^tG)=\chi_q(\det(E_3-Y_4)).$$
Since $Y_5,Y_6$ are unconstrained modulo $q$, 
$$c_d(M_{d'},I)=q^{d(n-d')+(n-d')(d'-d)+(n-d')(n-d'+1)/2}\bbeta(d',d)\sym_q^{\chi}(d'-d)$$
(recall $Y_6$ is symmetric).
Collecting terms and setting $t=d'-d$ yields the result.
\end{proof}

To help us diagonalize the space Eisenstein series of square-free level,
we put a partial ordering on $\{\sigma\}$, the multiplicative partitions of $\stufe$, 
as follows.  

\smallskip\noindent
{\bf Definition}  
Let $\sigma,\alpha$ be multiplicative partitions of $\stufe$, and let $q$ be a prime dividing $\stufe$.
We write $\sigma<\alpha\ (q)$ if $\rank_qM_{\sigma}<\rank_qM_{\alpha},$
$\sigma=\alpha\ (q)$ if $\rank_qM_{\sigma}=\rank_qM_{\alpha},$ and
$\sigma\le\alpha\ (q)$ if $\rank_qM_{\sigma}\le\rank_qM_{\alpha}.$
For $Q|\stufe$, we write $\sigma<\alpha\ (Q)$ if $\rank_qM_{\sigma}<\rank_qM_{\alpha}$ for all primes $q|Q$,
$\sigma=\alpha\ (Q)$ if $\rank_qM_{\sigma}=\rank_qM_{\alpha}$ for all primes $q|Q$,
$\sigma\le\alpha\ (Q)$ if $\rank_qM_{\sigma}\le\rank_qM_{\alpha}$ for all primes $q|Q$,

We first determine how to find eigenforms for $T(q)$.

\begin{cor}\label{Corollary 4.2}  
Suppose $\sigma$ is a  multiplicative partition of $\stufe$
so that $\E_{\sigma}\not=0$, and let $q$ be a prime dividing $\stufe$.  
For partitions $\alpha$ of $\stufe$ with $\alpha=\sigma\ (\stufe/q)$, $\alpha>\sigma\ (q)$, there are
$a_{\sigma,\alpha}(q)\in\C$ so that
$$\E_{\sigma}+\sum_{\substack{\alpha=\sigma\,(\stufe/q) \\ \alpha>\sigma\,(q)}}  a_{\sigma,\alpha}(q)\,\E_{\alpha}$$
is an eigenform for $T(q)$,
and  $a_{\sigma,\alpha}(q)\not=0$ only if either
(1) $\chi_q=1$, or (2) $\chi_q^2=1$ and $\rank_qM_{\alpha}-\rank_qM_{\sigma}$ is even.
With such $a_{\sigma,\alpha}$ and $d=\rank_qM_{\sigma}$, the eigenvalue of 
$\E_{\sigma}+\sum_{\substack{\alpha=\sigma\,(\stufe/q) \\ \alpha>\sigma\,(q)}}  a_{\sigma,\alpha}(q)\,\E_{\alpha}$
is
$$\lambda_{\sigma}(q)=q^{kd-d(d+1)/2} 
\chi_{\stufe/q}(\overline q X_dM_{\sigma}, X_d)$$
where $q\overline q\equiv 1\ (\stufe/q)$.
\end{cor}

\begin{proof}  
By Lemma 6.6
$\sym_q^{\chi}(t)=0$ if and only if (1) $\chi_q=1$, or (2) $\chi_q^2=1$ and $t$ is even.
Thus by Theorem 4.1, the subspace
\begin{align*}
\spn\Big\{\E_{\alpha}:\ \alpha
&=\sigma\ (\stufe/q),\ \alpha\ge\sigma\ (q),\ \E_{\alpha}\not=0,
\text{ and either } 
(1)\ \chi_q=1,\\
& \text{ or }  (2)\ \chi_q^2=1 \text{ and }\rank_qM_{\alpha}-\rank_qM_{\sigma} 
\text{ is even }\Big\}
\end{align*}
is invariant under $T(q)$, and the matrix for $T(q)$ on this subspace basis
(ordered with $\rank_qM_{\alpha}$ increasing) is upper triangular with 
diagonal entries $\lambda_{\alpha}(q)$.  Then the standard process of diagonalizing an upper triangular matrix yields
the result.  
\end{proof}

We now diagonalize the space of Eisenstein series with respect to
$$\{T(q):\ q \text{ prime, } q|\stufe\ \}$$ 
and obtain a multiplicity-one result for the Eisenstein series of square-free level.

\begin{cor}\label{Corollary 4.3}  Suppose $\sigma$ a multiplicative partition of $\stufe$ so that $\E_{\sigma}\not=0$.
For a prime $q|\stufe$ and 
$\alpha$ a multiplicative partition of $\stufe$ with
$\alpha\ge\sigma\ (\stufe)$, set 
$a_{\sigma,\alpha}(q)=1$ if $\alpha=\sigma\ (q)$, and otherwise set
$a_{\sigma,\alpha}(q)=a_{\rho,\alpha}(q)$
where $\rho$ is a multiplicative partition of $\stufe$ with
$\rho=\alpha\ (\stufe/q)$, $\rho=\sigma\ (q)$, and $a_{\rho,\alpha}(q)$ is as in Corollary 4.2.
For $Q|\stufe$ and $\alpha\ge\sigma\ (Q),$
set
$$a_{\sigma,\alpha}(Q)=\prod_{\substack{q|Q\\q \text{ prime}}} a_{\sigma,\alpha}(q).$$
Then with $$\widetilde \E_{\sigma}=\sum_{\alpha\ge\sigma\ (\stufe)} a_{\sigma,\alpha}(\stufe)\E_{\alpha},$$
for every prime $q|\stufe$ we have $\widetilde\E_{\sigma}|T(q)=\lambda_{\sigma}(q)\widetilde\E_{\sigma}$
(where $\lambda_{\sigma}(q)$ is as in Corollary 4.2).  Further, for $\sigma\not=\rho\ (\stufe)$, there is some
prime $q|\stufe$ so that $\lambda_{\sigma}(q)\not=\lambda_{\rho}(q).$
\end{cor}

\begin{proof}  Fix a prime $q|\stufe$.  For $\alpha, \beta$ multiplicative partitions of $\stufe$ with
$\alpha\ge\sigma\ (\stufe)$,  $\beta=\alpha\ (\stufe/q)$, and
$\beta=\sigma\ (q)$, we have $a_{\sigma,\alpha}(\stufe)=a_{\sigma,\beta}(\stufe/q)a_{\beta,\alpha}(q).$
Thus, varying $\beta,\alpha$ so that $\beta\ge\sigma\ (\stufe/q)$, $\beta=\sigma\ (q)$,
$\alpha=\beta\ (\stufe/q)$, $\alpha\ge \beta\ (q),$ we have
$$\widetilde\E_{\sigma}=\sum_{\beta}a_{\sigma,\beta}(\stufe/q)\sum_{\alpha}a_{\beta,\alpha}(q)\E_{\alpha}.$$
By Corollary 4.2,
$$\sum_{\alpha}a_{\beta,\alpha}(q)\E_{\alpha}|T(q)=
\lambda_{\beta}(q)\sum_{\alpha}a_{\beta,\alpha}(q)\E_{\alpha}.$$
So to show $\widetilde\E_{\sigma}|T(q)=\lambda_{\sigma}(q)\widetilde\E_{\sigma},$
we need to show that $\lambda_{\beta}(q)=\lambda_{\sigma}(q)$ for any $\beta$ so that
$\beta\ge \sigma\ (\stufe/q)$, $\beta=\sigma\ (q),$ and $a_{\sigma,\beta}(\stufe/q)\not=0.$
Equivalently, we need to show that for $\beta\ge \sigma\ (\stufe/q)$, $\beta=\sigma\ (q)$ with
$a_{\sigma,\beta}(\stufe/q)\not=0$, we have
$$\chi_{q'}\left(\overline q X_d M_{\beta},X_d\right)
=\chi_{q'}\left(\overline q X_d M_{\sigma},X_d\right)$$
for all primes $q'|\stufe/q$ (where $q\overline q\equiv 1\ (\stufe/q)$).

Let $d=\rank_qM_{\sigma}$, and fix $\beta$ so that $\beta\ge\sigma\ (\stufe/q)$, 
$\beta=\sigma\ (q),$ and $a_{\sigma,\beta}(\stufe/q)\not=0.$
Let $q'$ be a prime dividing $\stufe/q$.  If $\beta=\sigma\ (q')$, then
$M_{\beta}\equiv M_{\sigma}\ (q')$ and so
$$\chi_{q'}\left(\overline q X_d M_{\beta},X_d\right)
=\chi_{q'}\left(\overline q X_d M_{\sigma},X_d\right).$$
So suppose $\beta>\sigma\ (q').$  Since $a_{\sigma,\beta}(\stufe/q)\not=0$, by Corollary 5.3
we either have $\chi_{q'}=1$, or $\chi_{q'}^2=1$ with $\rank_{q'}M_{\beta}$, $\rank_{q'}M_{\sigma}$ of the same parity.
Consequently (using Proposition 3.7),
$$\chi_{q'}\left(\overline q X_d M_{\beta},X_d\right)
=\chi_{q'}\left(\overline q X_d M_{\sigma},X_d\right).$$
Hence
$\widetilde\E_{\sigma}|T(q)=\lambda_{\sigma}(q)\widetilde\E_{\sigma}$, proving
the first part of the corollary.

To prove the second part, suppose now that $\sigma\not=\rho\ (\stufe).$  Thus for some prime $q|\stufe$,
we have $d=\rank_qM_{\sigma}\not=\rank_qM_{\rho}=d'$.  Then
$$|\lambda_{\sigma}(q)|=q^{kd-d(d+1)/2}\not=q^{kd'-d'(d'+1)/2}=|\lambda_{\rho}(q)|,$$
since $0\le d,d'\le n$ and $k>n+1$.  
\end{proof}

Now we evaluate the action of $T_j(q^2)$ on $\E_{\sigma}$.  Note that since the Hecke operators commute,
the multiplicity-one result of Corollary 4.3 tells us that each $\widetilde\E_{\sigma}$ is an eigenform for
$T_j(q^2)$ ($1\le j\le n$), and in fact for $T(p)$, $T_j(p^2)$ ($1\le j\le n)$) for any prime $p$.
So we could simply do enough computation to find the eigenvalue $\lambda_{j;\sigma}(q^2)$, but
we take just a bit more effort and give a complete description of $\E_{\sigma}|T_j(q^2)$.
Then in Corollary 4.5 we simplify our expressions for the  $T_j(q^2)$-eigenvalues.

\begin{thm}\label{Theorem 4.4}  Assume $\stufe$ is square-free, and fix a prime $q|\stufe$.  For
$\sigma$ a multiplicative partition of $\stufe/q$ and $0\le d\le n$, let $\E_{\sigma_d}$
be the level $\stufe$ Eisenstein series as in Theorem 4.1; suppose $\E_{\sigma_d}\not=0$.
Then for $0\le j\le n$, 
$$\E_{\sigma_d}|T_j(q^2)=\sum_{t=0}^{n-d} A_j(d,t)\E_{\sigma_{d+t}}$$
where
\begin{align*}
A_j(d,t)&=q^{(j-t)d-t(t+1)/2}\bbeta_q(d+t,t)\\
&\quad \cdot \sum_{d_1=0}^j\sum_{d_5=0}^{j-d_1}\sum_{d_8=0}^{d_5}
q^{a_j(d;d_1,d_5,d_8)}
\overline\chi_{\stufe/q}(X_{d_1,r}^{-1}M_{\sigma_{d}}X_j,X_{d_1,r}^{-1}X_j^{-1})    \\
&\quad \cdot \bbeta_q(d,d_1)\bbeta_q(t,d_5)\bbeta_q(n-d-t,d_1+n-d-j-d_8)\\
&\quad \cdot \bbeta_q(t-d_5,d_8)\sym_q^{\chi}(t-d_5-d_8)\sym_q^{\chi}(d_5,d_8),
\end{align*}
$r=j-d_1-d_5+d_8$,
and 
\begin{align*}
a_j(d;d_1,d_5,d_8)
&=(k-d)(2d_1+d_5-d_8)+d_1(d_1-d_8-j-1)\\
&\quad +d_8(j-d_5)-d_5(d_5+1)/2+d_8(d_8+1)/2.
\end{align*}
(Here $\sym_q^{\chi}(b,c)$ is as defined in \S 2.)
Thus $\widetilde \E_{\sigma_d}|T_j(q^2)=A_j(d,0)\widetilde\E_{\sigma_d}.$
\end{thm}

\begin{proof}  As in the proof of Theorem 4.1, temporarily write $\E_{d'}$ for
$\E_{\sigma_{d'}}$ and $M_{d'}$ for $M_{\sigma_{d'}}$.
Let $\K_{r,s}=\K_{r,s}(q)$, $\Y_{j,0}=\Y_{j,0}(q^2)$, $X_{r,s}=X_{r,s}(q)$,
$\bbeta(m,r)=\bbeta_q(m,r)$, $\bmu(m,r)=\bmu_q(m,r)$.

By Proposition 2.1,
$$\E_{d}|T_j(q^2)
=q^{j(k-n-1)}\sum_{G,Y}\E_{d}|\begin{pmatrix} X_j^{-1}\\&X_j\end{pmatrix}
\begin{pmatrix} G^{-1}&Y\,^tG\\&^tG\end{pmatrix}$$
where $G$ varies over $SL_n(\Z)/\K_j$, $Y$ over $\Y_{j,0}$;
recall that we can take $G\equiv I\ (\stufe/q)$ and $Y\equiv 0\ (\stufe/q)$.
So
\begin{align*}
&\E_{d}(\tau)|T_j(q^2)\\
&=q^{j(k-n-1)}\sum_{G,Y,M,N} \overline\chi(M,N)
\det\left(MX_j^{-1}G^{-1}\tau+MX_j^{-1}Y\,^tG+NX_j\,^tG\right)^{-k}\end{align*}
(where $SL_n(\Z)(M\ N)$ varies over $SL_n(\Z)(M_{d}\ I)\Gamma_0(\stufe)$).

Take $(M\ N)\in SL_n(\Z)(M_d\ I)\Gamma_0(\stufe)$.
Let $d_1$ be the $q$-rank of the first $j$ columns of $M$ (so $d_1\le j$);
using left-multiplication from $SL_n(\Z)$, we can
adjust our choice of representative to assume $M=\begin{pmatrix} M_1&M_2\\ qM_3&M_4\\qM_5'&qM_6'\end{pmatrix}$
where $M_1$ is $d_1\times j$ (so $\rank_qM_1=d_1$), $M_4$ is $d_4\times(n-j)$ with $\rank_qM_4=d_4
=d-d_1$.
Correspondingly, write $N=\begin{pmatrix} N_1&N_2\\N_3&N_4\\N_5'&N_6'\end{pmatrix}$ where $N_1$ is $d_1\times j$
and $N_4$ is $d_4\times(n-j)$.  
Take $r$ so that $\rank_q\begin{pmatrix} M_1&0\\M_5'&N_6'\end{pmatrix} = n-d_4-r$; so adjusting our choice of representative, we 
can assume 
$$(qM_5'\ qM_6'\ N_5'\ N_6')=\begin{pmatrix} qM_5&qM_6&N_5&N_6\\q^2M_7&qM_8&N_7&qN_8\end{pmatrix}$$
where $M_6, N_6$ are $(n-d-r)\times(n-j)$ and $\rank_q\begin{pmatrix} M_1&0\\M_5&N_6\end{pmatrix}=n-d_4-r.$
Note that since $(M,N)=1$, we must have $\rank_qN_7=r.$  Then 
$$X_{d_1,r} (M\ N)\begin{pmatrix} X_j^{-1}\\&X_j\end{pmatrix}=\begin{pmatrix} M_1&qM_2&q^2N_1&qN_2\\M_3&M_4&qN_3&N_4\\M_5&qM_6&qN_5&N_6\\M_7&M_8&N_7&N_8\end{pmatrix}$$
has $q$-rank $n$.  Hence for any $Y\in\Y_j$,
$$(M'\ N')=X_{d_1,r} (M\ N)\begin{pmatrix} X_j^{-1}\\&X_j\end{pmatrix}
\begin{pmatrix} G^{-1}&Y\,^tG\\0&^tG\end{pmatrix}$$
is a coprime symmetric pair with $\rank_qM'=d+t$ for some $t\ge 0$.
Note that $\det(M'\tau+N')^{-k}=q^{-k(d_1-r)}\det(MX_j^{-1}G^{-1}\tau+MX_j^{-1}Y\,^tG+NX_j\,^tG)^{-k}.$

As discussed in the proof of Theorem 4.1, we have
$$\E_{d}|T_j(q^2)=\sum_{d'=d}^{n}c_d(M_{d'})\E_{d'},$$ some $c_d(M_{d'})\in\C$.
So reversing, suppose $d'\ge d$ and $\E_{d'}\not=0$.
To compute $c_d(M_{d'})$,
we need to identify the equivalence classes $SL_n(\Z)(M\ N)\in SL_n(\Z)(M_d\ I)\Gamma_0(\stufe)$
and $Y\in\Y_{j,0}$, $G\in SL_n(\Z)/\K_j$ so that
$$X_{d_1,r}(M\ N)\begin{pmatrix} X_j^{-1}\\&X_j\end{pmatrix}
\begin{pmatrix} G^{-1}&Y\,^tG\\0&^tG\end{pmatrix}\in SL_n(\Z)(M_{d'}\ I).$$
Equivalently, we need to identify $Y\in\Y_{j,0}$, $G\in SL_n(\Z)/\K_j$ and the equivalence classes
$$SL_n(\Z)X_{d_1,r}^{-1}E(M_{d'}\ I)\begin{pmatrix} G&-GY\\&^tG^{-1}\end{pmatrix}\begin{pmatrix} X_j\\&X_j^{-1}\end{pmatrix}
\in SL_n(\Z)(M_d\ I)\Gamma_0(\stufe)$$
where $E\in SL_n(\Z)$.
For $E\in SL_n(\Z)$, we have $X_{d_1,r}^{-1}EX_{d_1,r}\in SL_n(\Z)$ if and only if $E\in\K_{d_1,r},$
so we only need to consider $E\in\K_{d_1,r}\backslash SL_n(\Z)$.
Thus we need to consider all $E,G,Y$ so that with 
$$(M\ N)=X_{d_1,r}^{-1}E(M_{d'}\ I)\begin{pmatrix} G&-GY\\&^tG^{-1}\end{pmatrix}\begin{pmatrix} X_j\\&X_j^{-1}\end{pmatrix},$$
$M,N$ are integral with $(M,N)=1$ and $\rank_qM=d$
(that $M\,^tN$ is symmetric is automatic).
Note that since we can take $E,G\equiv I\ (\stufe/q)$ and $Y\equiv0\ (\stufe/q)$ and we know
$M_{d'}\equiv M_d\ (\stufe/q)$, for such $(M\ N)$ we have
$$\chi_{\stufe/q}(M,N)=\chi_{\stufe/q}(X_{d_1,r}^{-1}M_dX_j,X_{d_1,r}^{-1}X_j^{-1}).$$

For $E,G\in SL_n(\Z)$, write
$$EM_{d'}G=\begin{pmatrix}M_1&M_2\\M_3'&M_4'\\M_7&M_8\end{pmatrix},
\ E\,^tG^{-1}=\begin{pmatrix} N_1&N_2\\N_3'&N_4'\\N_7&N_8\end{pmatrix}$$
where $M_1, N_1$ are $d_1\times j$, $M_7, N_7$ are $r\times j$.
Then $$M=\begin{pmatrix} M_1&M_2/q\\qM_3'&M_4'\\q^2M_7&qM_8\end{pmatrix}.$$
So to have $M$ integral, we need $M_2\equiv0\ (q)$, and to have $\rank_qM=d$, we need
$\rank_q\begin{pmatrix}M_1&0\\0&M_4'\end{pmatrix}=d.$  So suppose these conditions are met.
We have $Y=\begin{pmatrix}U&V\\^tV&0\end{pmatrix}$ where $U$ is $j\times j$ and symmetric; to have $N$ integral,
we need $N_1\equiv M_1U+M_2\,^tV\ (q^2),$ $N_2\equiv M_1V\ (q),$ and $N_3'\equiv M_3'U+M_4'\,^tV\ (q).$
We are supposing that $\rank_q(M_1\ M_2\ N_1\ N_2)=d_1$ and $M_2\equiv0\ (q)$, so we can solve
these first two congruences only if $\rank_qM_1=d_1$.  So supposing this condition is met, 
we have $M_2/q$ in the column span of $M_1$ modulo $q$, so we must have $\rank_qM_4'=d_4$ where $d_4=d-d_1$.
Then adjusting $E$ using left multiplication from $\K_{d_1,r}$, and adjusting $G$ using right multiplication from $\K_j$,
we can assume 
$$EM_{d'}G=\begin{pmatrix} M_1&M_2\\M_3&M_4\\M_5&M_6\\M_7&M_8\end{pmatrix}$$
where $M_4$ is $d_4\times (n-j)$ with $\rank_qM_4=d_4$ and $M_6\equiv0\ (q)$; further, we can assume
$M_i=(A_i'\ A_i)$ where $A_i'$ has $d_1$ columns when $i$ is odd, $d_4$ columns when $i$ is even,
$A_i'\equiv 0\ (q^2)$ for $i\not=1,4$, and $A_i\equiv 0\ (q^2)$ for $i\le 4$.
Correspondingly, split $N_3'$ as $\begin{pmatrix} N_3\\N_5\end{pmatrix}$, $N_4'$ as $\begin{pmatrix}N_4\\N_6\end{pmatrix}$ where
$N_3, N_4$ have $d_4$ rows, and split $N_i$ as $(B_i'\ B_i)$ where $B_i'$ has $d_1$ columns when $i$ is odd, $d_4$ columns when
$i$ is even.  Split $U$ as $\begin{pmatrix}U_1&U_2\\^tU_2&U_3\end{pmatrix}$ where $U_1$ is $d_1\times d_1$,
and split $V$ as $\begin{pmatrix}V_1&V_2\\V_3&V_4\end{pmatrix}$ where $V_1$ is $d_1\times d_4$.  
Then $M_1U\equiv A_1'(U_1\ U_2)\ (q^2)$, $M_1V\equiv A_1'(V_1\ V_2)\ (q),$ $M_4\,^tV\equiv A_4'(\,^tV_1\ ^tV_3)\ (q).$
So to have $N$ integral, we need to choose $U_1, U_2, V_1, V_2, V_3$ so that
$$(B_1'\ B_1)\equiv A_1'(U_1\ U_2)\ (q^2),\ (B_2'\ B_2)\equiv A_1'(V_1\ V_2)\ (q),\ B_3\equiv A_4'\,^tV_3\ (q).$$
Then by the symmetry of $EM_{d'}\,^tE$, we have $B_3'\,^tA_1'\equiv A_4'\,^tB_2'\ (q),$
so we have $B_3'\equiv A_4'\,^tV_1\ (q).$ 
By symmetry, we also have
$$B_5'\,^tA_1'\equiv A_5\,^tB_1+A_6\,^tB_2\equiv A_5\,^tU_2\,^tA_1'+A_6\,^tV_2\,^tA_1'\ (q^2),$$
$$B_6'\,^tA_4'\equiv A_5\,^tB_3\equiv A_5V_3\,^tA_4'\ (q),$$
$$B_7'\,^tA_1'\equiv A_7\,^tB_1+A_8\,^tB_2\equiv A_7\,^tU_2\,^tA_1'+A_8\,^tV_2\,^tA_1'\ (q).$$
So to have $N$ integral, we also need to choose $U_3$ so that $B_5\equiv A_5U_3\ (q)$, and then the lower
$n-d$ rows of $N$ are congruent modulo $q$ to
$$\begin{pmatrix} 0&(B_5-A_5U_3-A_6\,^tV_4)/q&0&B_6-A_5V_4\\0&B_7-A_7U_3-A_8\,^tV_4&0&0\end{pmatrix}.$$

Further refining our choices for $E, G$ using $\K_{d_1,r},\K_j$, we can assume
$$A_5\equiv\begin{pmatrix} \alpha_5&0&0\\0&q\alpha'_5&0\end{pmatrix}\ (q^2),\ 
A_6\equiv\begin{pmatrix} 0&0\\0&q\alpha'_6\end{pmatrix}\ (q^2),$$
$$A_7\equiv\begin{pmatrix} 0&0&0\\0&0&\alpha_7\\0&0&0\end{pmatrix}\ (q),\ 
A_8\equiv\begin{pmatrix} 0&0\\0&0\\\alpha_8&0\end{pmatrix}\ (q)$$
where $\alpha_i$ is $d_i\times d_i$ and invertible modulo $q$,
$\alpha'_5$ is $(n-d-r-d_5)\times(j-d_1-d_5-d_7)$, and
$\alpha'_6$ is $(n-d-r-d_5)\times(n-j-d_4-d_8)$; here
the top $r-d_7-d_8$ and bottom $d_8$ rows of $A_7$ are 0 modulo $q$.
Correspondingly, write 
$$B_5=\begin{pmatrix} \beta_1&\beta_2&\beta_3\\ \beta_4&\beta_5&\beta_6\end{pmatrix},\ 
B_6=\begin{pmatrix} \gamma_1&\gamma_2\\ \gamma_3&\gamma_4\end{pmatrix},$$
$$B_7=\begin{pmatrix} \delta_1&\delta_2&\delta_3\\ \delta_4&\delta_5&\delta_6\\ \delta_7&\delta_8&\delta_9\end{pmatrix},\ 
B_8=\begin{pmatrix} \epsilon_1&\epsilon_2\\ \epsilon_3&\epsilon_4\\ \epsilon_5&\epsilon_6\end{pmatrix}.$$
By symmetry and the invertibility of $\alpha_5, \alpha_7, \alpha_8$ modulo $q$, we have that
$\beta_4$, $\beta_6$, $\gamma_3$, $\delta_1$, $\delta_3$, $\epsilon_1\equiv 0\ (q)$, the bottom $n-d-r-d_5$ rows of $B_5', B_6'$,
and the top $r-d_7-d_8$ rows of $B_7', B_8'$ are  0 modulo $q$.
Then since $E\,^tG^{-1}$ is invertible, we know that $\rank_q\begin{pmatrix}\beta_5&\gamma_4\\ \delta_2&\epsilon_2\end{pmatrix}=n-d'.$

Write 
$$U_3=\begin{pmatrix} \mu_1&\mu_2&\mu_3\\^t\mu_2&\mu_4&\mu_5\\^t\mu_3&^t\mu_5&\mu_6\end{pmatrix},\ 
V_4=\begin{pmatrix} \nu_1&\nu_2\\ \nu_3&\nu_4\\ \nu_5&\nu_6\end{pmatrix}$$
where $\mu_1$ is $d_5\times d_5$, $\mu_6$ is $d_7\times d_7$, $\nu_1$ is $d_5\times d_8$, $\nu_5$ is $d_7\times d_8$.
To have $B_5\equiv A_5U_3\ (q),$ we need $(\beta_1\ \beta_2\ \beta_3)\equiv\alpha_5(\mu_1\ \mu_2\ \mu_3)\ (q),$
and $\beta_5\equiv0\ (q)$ (and hence $\gamma_4$ is invertible modulo $q$). 
When these conditions are met, we must have $\rank_q\gamma_4=n-d-r-d_5$, and by symmetry,
$$\delta_4\,^t\alpha_5\equiv\alpha_7\,^t\beta_3\equiv\alpha_7\,^t\mu_3\,^t\alpha_5\ (q).$$
Then to have $(M,N)=1$, we need 
$$B=\begin{pmatrix} (B_5-A_5U_3-A_6\,^tV_4)/q&B_6-A_5V_4\\B_7-A_7U_3-A_8\,^tV_4&q(B_8-A_7V_4)\end{pmatrix}$$
to have $q$-rank $n-d$.  Note that modulo $q$, $B$ is congruent to
$$\begin{pmatrix}
 (\beta_1-\alpha_5\mu_1)/q&(\beta_2-\alpha_5\mu_2)/q&(\beta_3-\alpha_5\mu_3)/q&\gamma_1-\alpha_5\nu_1&\gamma_2-\alpha_5\nu_2\\
0&*&*&0&\gamma_4\\
0&\delta_2&0&0&0\\
0&\delta_5-\alpha_7\,^t\mu_5&\delta_6-\alpha_7\mu_6&0&0\\
\delta_7-\alpha_8\,^t\nu_1&\delta_8-\alpha_8\,^t\nu_3&\delta_9-\alpha_8\,^t\nu_6&0&0\end{pmatrix}.$$

Since $E\,^tG^{-1}$ is invertible, and given that $\beta_4, \beta_5, \beta_6, \gamma_3$ and the lower $n-d-r-d_5$ rows of
$B_5', B_6'$ are 0 modulo $q$, we must have $\rank_q\gamma_4=n-d-r-d_5$.  To have $B$ invertible modulo $q$,
we need $\rank_q\delta_2=r-d_7-d_8$.  Given the sizes of $\gamma_4,\delta_2$, this requires
$$n-d-r-d_5\le n-j-d_4-d_8 \text{ and }  r-d_7-d_8\le j-d_1-d_5-d_7,$$
so this requires $r=j-d_1-d_5+d_8$ (in which case $\gamma_4,\delta_2$ are square, and hence invertible modulo $q$).

Choose $(n-d)\times(n-d)$ permutation matrices $P_1, P_2$ so that 
$$P_1\begin{pmatrix}A_5&A_6\\A_7&A_8\end{pmatrix} P_2\equiv\begin{pmatrix}\alpha_5\\&\alpha_8\\&&0\\&&&\alpha_7\\&&&&0\end{pmatrix}\ (q),$$
$$P_1\begin{pmatrix}B_5&B_6\\B_7&B_8\end{pmatrix} P_2
=\begin{pmatrix} \beta_1&\gamma_1&\gamma_2&\beta_3&\beta_2\\
\delta_7&\epsilon_5&\epsilon_6&\delta_9&\delta_8\\
\beta_4&\gamma_3&\gamma_4&\beta_6&\beta_5\\
\delta_4&\epsilon_3&\epsilon_4&\delta_6&\delta_5\\
\delta_1&\epsilon_1&\epsilon_2&\delta_3&\delta_2\end{pmatrix}.$$
(So $P_1$ corresponds to the permutation $(2\ 3\ 5)$, $P_2$ to the permutation $(2\ 5\ 3\ 4)$.)
Thus (still supposing that $\beta_5\equiv0\ (q)$, and that $\gamma_4,\delta_2$ are invertible modulo $q$), we have
$P_1BP_2$ is congruent modulo $q$ to
$$\begin{pmatrix} 
(\beta_1-\alpha_5\mu_1)/q&\gamma_1-\alpha_5\nu_1&\gamma_2-\alpha_5\nu_2&(\beta_3-\alpha_5\mu_3)/q&(\beta_2-\alpha_5-\mu_2)/q\\
\delta_7-\alpha_8\,^t\nu_1&0&0&\delta_9-\alpha_8\,^t\nu_6&\delta_8-\alpha_8\,^t\nu_3\\
0&0&\gamma_4&*&*\\
0&0&0&\delta_6-\alpha_7\mu_6&\delta_5-\alpha_7\,^t\mu_5\\
0&0&0&0&\delta_2\end{pmatrix}.$$
Hence $B$ is invertible modulo $q$ if and only if 
$\begin{pmatrix} (\beta_1-\alpha_5\mu_1)/q&\gamma_1-\alpha_5\nu_1\\ \delta_7-\alpha_8\,^t\nu_1&0\end{pmatrix}$
and $\delta_6-\alpha_7\mu_6$
are invertible modulo $q$.  Note that by the symmetry of $M_{d'}$, we know that
$(\delta_6-\alpha_7\mu_6)\,^t\alpha_7$ and
$\begin{pmatrix} (\beta_1-\alpha_5\mu_1)\,^t\alpha_5/q&(\gamma_1-\alpha_5\nu_1)\,^t\alpha_8\\ 
(\delta_7-\alpha_8\,^t\nu_1)\,^t\alpha_5&0\end{pmatrix}$
are symmetric modulo $q$.

To compute $\chi(M,N),$ recall that we can (and do) assume that $E,G\equiv I\ (\stufe/q)$, $Y\equiv0\ (\stufe/q)$,
and we know that $M_{d'}\equiv M_d\ (\stufe/q)$;
so $$\chi_{\stufe/q}(M,N)=\chi_{\stufe/q}(X_{d_1,r}^{-1}M_dX_j, X_{d_1,r}^{-1}X_j^{-1}).$$
To help compute $\overline\chi_q(M,N)$, let $G_1$ be the $n\times n$ permutation matrix so that
$$EM_{d'}GG_1\equiv\begin{pmatrix}A_1'&0&0&0\\0&A_4'&0&0\\0&0&A_5&A_6\\0&0&A_7&A_8\end{pmatrix}\ (q).$$
Setting $E_1=\begin{pmatrix}I_d\\&P_1\end{pmatrix}$, $G_2=\begin{pmatrix} I_d\\&P_2\end{pmatrix},$ and remembering that $^tP^{-1}=P$ for
a permutation matrix $P$,
we have
\begin{align*}
&\chi_q(\det E_1G_1G_2)\chi_q(M,N)\\
&\quad = \chi_q(E_1MG_1G_2,E_1NG_1G_2)\\
&\quad = \overline\chi_q(\det A_1'\cdot\det A_4'\cdot\det\alpha_5\cdot\det\alpha_7\cdot\det\alpha_8)
\chi_q(\det\gamma_4\cdot\det\delta_2)\\
&\qquad\cdot 
\chi_q\left(\det\begin{pmatrix} (\beta_1-\alpha_5\mu_1)\,^t\alpha_5/q&(\gamma_1-\alpha_5\nu_1)\,^t\alpha_8
\\(\delta_7-\alpha_8\,^t\nu_1)\,^t\alpha_5&0\end{pmatrix}\cdot \det(\delta_6-\alpha_7\mu_6)\,^t\alpha_7\right).
\end{align*}
Also, since $\chi_q(M_{d'},I)=1$, we have
\begin{align*}
&\chi_q(\det E_1G_1G_2)\\
&\quad=\chi_q(E_1M_{d'}G_1G_2,E_1G_1G_2)\\
&\quad= \overline\chi_q(\det A_1'\cdot\det A_4'\cdot\det\alpha_5\cdot\det\alpha_7\cdot\det\alpha_8)
\chi_q(\det\gamma_4\cdot\det\delta_2).\end{align*}

To summarise:  Given $(M'\ N')$ with $\rank_qM'=d'$, and given
choices for $d_1, d_4=d-d_1, d_5, d_7, d_8=d'-d-d_5-d_7$ and $r=j-d_1-d_5+d_8$
(with $d_1+d_5+d_7\le j$, $d_4+d_8\le n-j$),
to be able to choose $E,G,Y$ so that $M,N$ are integral and coprime with
$\rank_qM=d$, we need to choose $E\in\K_{d_1,r}\backslash SL_n(\Z)$ so that the $q$-rank of the
upper $d_1$ rows of $EM'$ is $d_1$, and the $q$-rank of the upper $n-r$ rows of $EM'$ is $d+d_5$ (where
$d_5\le j-d_1$; note this is only possible when $d'-d-d_5\le r$).
By Lemma 6.4 (b), we have
\begin{align*}
&\bbeta(d',d+d_5)\bbeta(n-d',n-r-d-d_5)\bbeta(d+d_5,d_1)\\
&\quad \cdot q^{(d+d_5)(r+d+d_5-d')+d_1(n-d-d_5)}
\end{align*}
choices for $E$.
Modifying $E$ using left multiplication from $\K_{d_1,r}$, we can assume the upper $d+d_5$ rows of $EM_{d'}$ have $q$-rank $d+d_5$.
We need to choose $G\in SL_n(\Z)/\K_j$ to meet various conditions (as detailed in the preceding discussion); choosing $G_0\in SL_n(\Z)$ so that
$$EM_{d'}G_0\equiv\begin{pmatrix} M_1&0&0&0\\0&C&0&0\\0&0&0&0\\0&0&C'&0\end{pmatrix}\ (q)$$
where $M_1$ is $d_1\times d_1$, $C$ is $(d_4+d_5)\times(d_4+d_5)$, $C'$ is $(d_7+d_8)\times(d_7+d_8)$
and $M_1, C, C'$ are invertible modulo $q$, Lemma 6.5 describes the conditions that
$$E(M_{d'}\ I)\begin{pmatrix} G\\&^tG^{-1}\end{pmatrix}$$
must meet, where $G=G_0G'\in SL_n(\Z)/\K_j$ (note that as $G'$ varies over $SL_n(\Z)/\K_j$, so does $G$).
By Lemma 6.5, we have
$$\bbeta(d_4+d_5,d_4)\bbeta(d_7+d_8,d_8)q^{(d_4+d_8)(j-d_1-d_5)-d_7d_8}$$
choices for $G$.  
Then with further adjustments to $E$ using left multiplication from $\K_{d_1,r}$ and to $G$ using right multiplication
from $\K_j$ (as described above),
using notation as above and writing $\mu_i=\mu'_i+q\mu''_i$,
we have that
$\mu'_1$, $\mu'_2$, $\mu'_3$ are uniquely determined modulo $q$, $\mu_4$, $ \mu_5$ are unconstrained modulo $q^2$, and
$\mu_2''$, $\mu_3''$, $\mu_6''$, $\nu_2$, $\nu_3$, $\nu_4$, $\nu_5$, $\nu_6$ are unconstrained modulo $q$.
Let $\F=\Z/q\Z$;
as $\mu_1''$, $\nu_1$, $\mu_6'$ vary modulo $q$,
$$\begin{pmatrix} (\beta_1-\alpha_5\mu_1)\,^t\alpha_5/q&(\gamma_1-\alpha_5\mu_1)\,^t\alpha_8\\
(\delta_7-\alpha_8\,^t\nu_1)\,^t\alpha_5&0\end{pmatrix}$$
varies over elements in $\F^{d_5+d_8,d_5+d_8}_{\sym}$ of the form $\begin{pmatrix}C&D\\^tD&0\end{pmatrix}$
with $C$ $d_5\times d_5$, and $(\delta_6-\alpha_7\mu_6')\,^t\alpha_7$ varies over
$\F^{d_7,d_7}_{\sym}$.  Hence as we vary $Y$ subject to these constraints, we have
\begin{align*}
&\sum_Y \chi_q(X_{d_1,r}^{-1}EM_{d'}GX_j, X_{d_1,r}^{-1}E(\,^tG^{-1}-M_{d'}GY)X_j^{-1}) \\
&\quad= q^{(j-d_1)(n-d_1-d_4+1)-d_5(j-d_1+d_8+1)-d_7(d_7+1)/2}
\sym_q^{\chi}(d_5,d_8)\sym_q^{\chi}(d_7).
\end{align*}

This yields a formula for $A_j(d,t)$; to simplify this formula,
note that 
$\bbeta(m,s)=\bbeta(m,m-s)$, 
so
\begin{align*}
&\bbeta(d_1+d_4+d_5,d_1)\bbeta(d',d_1+d_4+d_5)\bbeta(d_4+d_5,d_4)\\
&\ =
\frac{\bmu(d+d_5,d_1) \bmu(d+t,t-d_5) \bmu(d-d_1+d_5,d_5)}
  {\bmu(d_1,d_1) \bmu(t-d_5,t-d_5) \bmu(d_5,d_5)}
\frac{\bmu(t,d_5)}{\bmu(t,d_5)} \\
&\ =
\frac{\bmu(d+t,d_1+t) \bmu(t,d_5)} {\bmu(d_1,d_1) \bmu(t,t) \bmu(d_5,d_5)} \\
&\ =
\frac{\bmu(d+t,t) \bmu(d,d_1) \bmu(t,d_5)} {\bmu(t,t) \bmu(d_1,d_1) \bmu(d_5,d_5)} \\
&\ =
\bbeta(d+t,t) \bbeta(d,d_1) \bbeta(t,d_5)
\end{align*}
where $t=d'-d$.
We have the constraints that $r=j-d_1-d_5+d_8$,
$d=d_1+d_4$,  $t=d_5+d_7+d_8$, $d_1+d_5+d_7\le j$, 
$d_4+d_8\le n-j$, and $d_8\le d_5$.  
Taking $0\le d_1\le j$, $0\le d_5\le j-d_1$, and $0\le d_8\le d_5$, a summand in the final formula for
$A_j(d,t)$ is 0 if the other constraints on the $d_i$ are not met.
\end{proof}

\begin{cor}\label{Corollary 4.5}  Let $\sigma$ be a multiplicative partition of $\stufe$, and suppose $\E_{\sigma}\not=0$.
Then 
for a prime $q|\stufe$ and $d=\rank_qM_{\sigma}$, we have
$\widetilde\E_{\sigma}|T_j(q^2)=\lambda_{j;\sigma}(q^2)\widetilde\E_{\sigma}$ where
$\stufe_i'=\stufe_i/(q,\stufe_i)$ and 
$$\lambda_{j;\sigma}(q^2)=
q^{jd}\sum_{\ell=0}^j q^{\ell(2k-2d-j+\ell-1)}\chi_{_{\stufe'_0}}(q^{2\ell})\chi_{_{\stufe'_n}}(q^{2(j-\ell)})
\bbeta_q(d,\ell)\bbeta_q(n-d,j-\ell).$$
\end{cor}

\begin{proof}  Since
$T(q)$ and $T_j(q^2)$ commute,
by Corollary 4.3 and Theorem 4.4, we know that $\widetilde\E_{\sigma}$ is an eigenform
for $T_j(q^2)$ with eigenvalue $A_j(d,0)$.  
By Theorem 4.4, using $\ell$ in place of $d_1$, and noting that $\bbeta(m,r)=\bbeta(m,m-r)$,
we have
\begin{align*}
& A_j(d,0)\\
&\quad =q^{jd}\sum_{\ell=0}^j q^{\ell(2k-2d+\ell-j-1)}
\overline\chi_{\stufe/q}(X_{\ell,j-\ell}^{-1}M_{\sigma_d}X_j,X_{\ell,j-\ell}^{-1}X_j^{-1})\\
&\quad\quad\cdot \bbeta(d,\ell)\bbeta(n-d,j-\ell).
\end{align*}
Note that $\begin{pmatrix} X_j\\&X_j^{-1}\end{pmatrix}$ is congruent modulo $\stufe/q$ to an element
of $Sp_n(\Z)$.  Thus
$$\overline\chi_{\stufe/q}(X_{\ell,j-\ell}^{-1}M_{\sigma_d}X_j,X_{\ell,j-\ell}^{-1}X_j^{-1})
=\overline\chi_{\stufe/q}(X_{\ell,j-\ell}^{-1}M_{\sigma_d},X_{\ell,j-\ell}^{-1})\chi_{\stufe/q}(q^j).$$
Then we use Propositions 3.6 and 3.7 to evaluate
$\overline\chi_{\stufe/q}(X_{\ell,j-\ell}^{-1}M_{\sigma_d},X_{\ell,j-\ell}^{-1}).$
\end{proof}

\bigskip
\section{Hecke operators on Eisenstein series of arbitrary level}
\smallskip

Fix $\stufe\in\Z_+$ and $\chi$ a character modulo $\stufe$.
Assume that $k>n+1$, $\chi(-1)=(-1)^k$, and that either $\stufe>2$ or $k$ is even.
Let $\left\{\gamma_{\sigma}=\begin{pmatrix} I&0\\M_{\sigma}&I\end{pmatrix}\right\}_{\sigma}$ be a set of representatives
for $\Gamma_{\infty}\backslash Sp_n(\Z)/\Gamma_0(\stufe)$ 
so that when $\stufe$ is square-free, $M_{\sigma}$ is as in Proposition 3.5, and
let $\E_{\sigma}=\E_{\gamma_{\sigma}}$.

To more easily describe the action of Hecke operators on $\E_{\sigma}$, we define an action of
$\U_{\stufe}\times\U_{\times}$ on Eisenstein series where
$\U_{\stufe}=(\Z/\stufe\Z)^{\times}$.  Toward this, we have the following.

\begin{prop}\label{Proposition 5.1}  
Suppose $\gamma=\begin{pmatrix} I&0\\M&I\end{pmatrix}\in Sp_n(\Z)$, $v,w\in\Z$ with
$(vw,\stufe)=1$; set $(v,w)\cdot M=v\begin{pmatrix} w\\&I\end{pmatrix} M\begin{pmatrix} w\\&I\end{pmatrix}$ and
$(v,w)\cdot\gamma=\begin{pmatrix} I&0\\(v,w)\cdot M&I\end{pmatrix}.$
With $v'\equiv v\ (\stufe)$, $w'\equiv w\ (\stufe)$, we have $(v',w')\cdot\gamma\in(v,w)\cdot\gamma\Gamma(\stufe);$
hence $\E_{(v,w)\cdot\gamma}=\E_{(v',w')\cdot\gamma}.$  Further, suppose
$\gamma'=\begin{pmatrix} I&0\\M'&I\end{pmatrix}\in Sp_n(\Z)$ so that $\E_{\gamma}=\E_{\gamma'}$.
Then $\E_{(v,w)\cdot\gamma}=\E_{(v,w)\cdot\gamma'}$, so we have an action of the group $\U_{\stufe}\times\U_{\stufe}$
on 
$$\left\{\E_{\gamma}:\ \gamma=\begin{pmatrix} I&0\\M&I\end{pmatrix}\in Sp_n(\Z)\ \right\}.$$
When $\stufe$ is square-free, 
$$\E_{(v,w)\cdot \sigma}=\left(\overline\chi_{\stufe_n}(w^2)\prod_{0<d\le n}\overline\chi_{\stufe_d}(v^d)\right)\,\E_{\sigma}$$
where we write 
$\E_{\sigma}$ for $\E_{\gamma_{\sigma}}$ and
$\E_{(v,w)\cdot \sigma}$ for $\E_{(v,w)\cdot \gamma_{\sigma}}$ with $\gamma_{\sigma}$ chosen as in \S 4.
\end{prop}

\begin{proof}  Take $\gamma,\gamma',v,w,v',w'$ as in the statement of the proposition.
By Proposition 3.5 we have $(v',w')\cdot\gamma\in(v,w)\cdot\gamma\Gamma(\stufe)$, so by Proposition 3.2
we have $\E_{(v,w)\cdot\gamma}=\E_{(v',w')\cdot\gamma}.$  

Now, given the assumption that $\E_{\gamma}=\E_{\gamma'}$, there is some $G\in GL_n(\Z)$, 
$\delta=\begin{pmatrix} A&B\\C&D\end{pmatrix}\in \Gamma_0(\stufe)$ so that
$G(M\ I)\delta=(M'\ I)$ and $\chi(\det G\cdot\det D)=1$.
By Lemma 6.1, there is some $\delta'=\begin{pmatrix}A'&B'\\C'&D'\end{pmatrix}\in\Gamma_0(\stufe)$
so that $\delta'\equiv\delta\ (\stufe)$, $\delta'\equiv I\ (v)$; set
$\delta''=\begin{pmatrix}A'&B'/v\\vC'&D'\end{pmatrix}$ (so $\delta''\in\Gamma_0(\stufe)$).
Since $SL_n(\Z)$ maps onto $SL_n(\Z/\stufe\Z)$, we can find $E\in SL_n(\Z)$ so that 
$$E\equiv\begin{pmatrix} w\cdot\det G\\&I\end{pmatrix}G\begin{pmatrix}\overline w\\&I\end{pmatrix}\ (\stufe);$$
set $G'=\begin{pmatrix}\det G\\&I\end{pmatrix}E$.  Take $\begin{pmatrix}r&s\\t&u\end{pmatrix}\in SL_2(\Z)$
so that 
$$\begin{pmatrix}r&s\\t&u\end{pmatrix}\equiv\begin{pmatrix} w&0\\0&\overline w\end{pmatrix}\ (\stufe),$$
and set
$$\beta=\begin{pmatrix} r&&s\\&I_{n-1}&&0\\t&&u\\&0&&I_{n-1}\end{pmatrix}$$
(so $\beta\in\Gamma_0(\stufe)$).  Then
$$G' \big((v,w)\cdot M\ I)\big)\beta^{-1}\delta''\beta
\equiv \big((v,w)\cdot M'\ I)\big)\ (\stufe),$$
so $\big((v,w)\cdot M'\ I)\big)\in GL_n(\Z) \big((v,w)\cdot M\ I)\big) \Gamma_0(\stufe).$
Since $\chi(\beta^{-1}\delta''\beta)=\chi(\det D)$ and $\chi(\det G')=\chi(\det G),$
by Proposition 3.2 we have $\E_{(v,w)\cdot\gamma}=\E_{(v,w)\cdot\gamma'}.$

Now suppose $\stufe$ is square-free.  
For all primes $q|\stufe$, we have $\rank_q(v,w)\cdot M=\rank_q M,$
so $((v,w)\cdot M_{\sigma}\ I)\in SL_n(M_{\sigma}\ I)\Gamma_0(\stufe).$
Fix a prime $q|\stufe$ and take $d=\rank_qM_{\sigma}$.
Thus by Proposition 3.5, we have $\E_{(v,w)\cdot\sigma}=\chi((v,w)\cdot M_{\sigma},I) \E_{\sigma}.$
If $d=0$ then $\chi_q((v,w)\cdot M_{\sigma},I)=\chi_q(0,I)=1.$  If $0<d<n$ then
$$\chi_q((v,w)\cdot M_{\sigma},I)=\chi_q\left(\begin{pmatrix} vw^2\\&vI_{d-1}\\&&0\end{pmatrix},I\right)=\overline\chi_q(v^dw^2),$$
and since $\chi_q^2=1$, $\overline\chi_q(v^dw^2)=\overline\chi_q(v^d).$  If $d=n$ then $\chi_q((v,w)\cdot M_{\sigma},I)=\overline\chi_q(v^nw^2).$
\end{proof}

Suppose $\E_{\sigma}\not=0$.  We have $(1,-1)\cdot\gamma_{\sigma}=\gamma_{\pm}\gamma_{\sigma}\gamma_{\pm},$
and $\Gamma^+_{(1,-1)\cdot\gamma_{\sigma}}=\gamma_{\pm}\Gamma^+_{\gamma_{\sigma}}\gamma_{\pm},$
so $\chi$ is trivial on $\Gamma^+_{(1,-1)\cdot\gamma_{\sigma}}$ and hence $\E_{(1,-1)\cdot\sigma}\not=0$.
With $\Gamma_0(\stufe)=\cup_{\delta}\Gamma^+_{(1,-1)\cdot\gamma_{\sigma}}$ (disjoint), left multiplication by $\gamma_{\pm}$
gives us $\Gamma_0(\stufe)=\cup_{\delta}\Gamma^+_{\gamma_{\sigma}}\gamma_{\pm}\delta$ (disjoint).
Then
\begin{align*}
\E_{(1,-1)\cdot\sigma} &\sum_{\delta}\overline\chi(\delta)\E^*|\gamma_{\pm}\gamma_{\sigma}\gamma_{\pm}|\delta\\
&=\chi(\gamma_{\pm})\sum_{\delta}\overline\chi(\delta)\E^*|\gamma_{\sigma}|\gamma_{\pm}\delta\\
&=\E_{\sigma}.
\end{align*}
If $((1,-1)\cdot M_{\sigma}\ I)=E(M_{\sigma}\ I)\gamma$ for some $E\in SL_n(\Z)$ and $\gamma\in\Gamma_0(\stufe)$,
then by Proposition 3.2 we have $\E_{(1,-1)\cdot\sigma}=\chi(\gamma)\E_{\sigma}$, so from above we must have
$\chi(\gamma)=1$.  Thus if $(M\ N)\in SL_n(\Z)(M_{\sigma}\ I)\gamma'$ and $(M\ N)\in SL_n(\Z)((1,-1)\cdot M_{\sigma}\ I)\gamma''$
for $\gamma',\gamma''\in\Gamma_0(\stufe),$ we have $\chi(\gamma')=\chi(\gamma'').$

So with $\chi(M,N)=\chi(\gamma)$ where $(M\ N)\in SL_n(\Z)(M_{\sigma}\ I)\gamma$ or
$(M\ N)\in SL_n(\Z)((1,-1)\cdot M_{\sigma}\ I)\gamma$ for $\gamma\in SL_n(\Z)$, $\chi(M,N)$ is well-defined.  Also,
$\E_{\sigma}=\frac{1}{2}(\E_{\sigma}+\E_{(1,-1)\cdot\sigma}),$
a fact we will use in the proofs of Theorems 5.2 and 5.4.

\begin{thm}\label{Theorem 5.2}  Suppose $\E_{\sigma}\not=0$;
fix a prime $p\nmid\stufe$ and $\overline p$ so that $p\overline p\equiv 1\ (\stufe)$.  Then 
$$\E_{\sigma}|T(p)=
\sum_{r=0}^n \chi(p^{n-r})p^{k(n-r)-(n-r)(n+r+1)/2}\bbeta_p(n,r)\,\E_{(p,\overline p^{r})\cdot\sigma}.$$
\end{thm}

\begin{proof}  Write $\K_r$ for $\K_r(p)$, $X_r$ for $X_r(p)$, $\bbeta(m,r)$ for $\bbeta_p(m,r).$

When $\E_{\sigma'}\not=0$, set
$$\delta=\begin{pmatrix} I&0\\-M_{\sigma'}&I\end{pmatrix} \begin{pmatrix} I&0\\(p\overline p)^3M_{\sigma'}&I\end{pmatrix}.$$
So $\delta\in\Gamma(\stufe)$, and hence by Proposition 3.2, $\E_{\gamma_{\sigma}}=\E_{\gamma_{\sigma}\delta}.$
Thus we may replace $\gamma_{\sigma'}$
by $\gamma_{\sigma'}\delta$
(effectively, we may assume $p^3|M_{\sigma'}$).

Now,
by Proposition 2.1, we have
\begin{align*}
\E_{\sigma}(\tau)|T(p)
& =p^{kn-n(n+1)/2} \sum_{\substack{(M\,N)\\ r,G,Y}} \chi(p^{n-r})\overline\chi(M,N)\\
&\ \cdot\det(pMX_r^{-1}G^{-1}\tau+pMX_r^{-1}Y\,^tG+NX_r\,^tG)^{-k}.
\end{align*}
Here $SL_n(\Z)(M\ N)$ varies over $SL_n(\Z)(M_{\sigma}\ I)\Gamma_0(\stufe)$,
$0\le r\le n$, and for each $r$,
$G$ varies over $SL_n(\Z)/\K_r$, $Y$ varies over $\Y_r(p)$; recall that since
$p\nmid\stufe$, we can take $G\equiv I\ (\stufe)$, $Y\equiv 0\ (\stufe)$.
Write $M=(M_1'\ M_2')$, $N=(N_1'\ N_2')$ with $M_1', N_1'$ $n\times r$, and let
$s=\rank_p(M_1'\ N_2')$.  
We can use left multiplication from $SL_n(\Z)$ to
adjust our representative $(M\ N)$ to assume that
$$M=\begin{pmatrix} pM_1&M_2\\M_3&M_4\end{pmatrix},\ N=\begin{pmatrix} N_1&pN_2\\N_3&N_4\end{pmatrix}$$ where 
$M_3, N_3$ are $s\times r$; so $\rank_p(M_3\ N_4)=s$,
and $\rank_p(M_2\ N_1)=n-s$ since $\rank_q(M\ N)=n$.  Set
$$(M'G\ N'\,^tG^{-1})=X^{-1}_{n-s} (pMX_r^{-1}\ NX_r+pMX_r^{-1}Y);$$
so 
\begin{align*}
\rank_p(M'\ N')&=\rank_p(pX^{-1}_{n-s}MX_r^{-1}\ X_{n-s}^{-1}NX_r)\\
&=\rank_p\begin{pmatrix} M_1&M_2&N_1&N_2\\ M_3&pM_4&pN_3&N_4\end{pmatrix}\\
&=n,
\end{align*}
and hence $(M',N')=1$.  Note that
$$\det(M'\tau+N')^{-k}=p^{k(n-s)}\det(pMX_r^{-1}G^{-1}\tau+pMX_r^{-1}Y\,^tG+NX_r\,^tG)^{-k}.$$

We know 
\begin{align*}
\E_{\sigma}|T(p)&=\frac{1}{2}(\E_{\sigma}+\E_{(1,-1)\cdot\sigma})|T(p)\\
&=\frac{1}{2}\sum_{\sigma'} c_{\sigma,\sigma'}(\E_{\sigma'}+\E_{(1,-1)\cdot\sigma'})
\end{align*}
for some $c_{\sigma,\sigma'}\in\C$.
So to compute $c_{\sigma,\sigma'}$ for any given $\sigma'$, we first identify those $r,s,G,Y$ and
$$SL_n(\Z)(M\ N)\in GL_n(\Z)(M_{\sigma}\ I)\Gamma_0(\stufe)$$
so that
\begin{align*}
&X_{n-s}^{-1}(pMX_r^{-1}G^{-1}\ NX_r\,^tG+pMX_r^{-1}Y\,^tG) \\
&\quad \in SL_n(\Z)(M_{\sigma'}\ I)\cup SL_n(\Z)((1,-1)\cdot M_{\sigma'}\ I).
\end{align*}
Equivalently, we identify $r,s,G,Y$ and $SL_n(\Z)$-equivalence classes
$$SL_n(\Z)X_{n-s}E\left(\frac{1}{p}M'GX_r\ \,\left(^tG^{-1}-\frac{1}{p}M'GY\right)X_r^{-1}\right)$$
that lie in $GL_n(\Z)(M_{\sigma}\ I)\Gamma_0(\stufe)$, where $M'=M_{\sigma'}$ or $(1,-1)\cdot M_{\sigma'}$,
and $E\in SL_n(\Z).$
Note that we only need to consider $E\in\,^t\K_{n-s}\backslash SL_n(\Z)$, as
$SL_n(\Z)X_{n-s}E=SL_n(\Z)X_{n-s}$ if and only if $E\in\,^t\K_{n-s};$ note also that we can take $E\equiv I\ (\stufe).$

Take $M'=M_{\sigma'}$ or $(1,-1)\cdot M_{\sigma'}$ (some $\sigma'$).
Recall that we can assume $p^3|M_{\sigma'}$, so with
$$(M\ N)=X_{n-s}E\left(\frac{1}{p}M'GX_r\ \,\left(^tG^{-1}-\frac{1}{p}M'GY\right)X_r^{-1}\right),$$
we have $M\equiv 0\ (p)$, and we have $N$ integral with $\rank_pN=n$ if and only if $n-s=r$ and $E\,^tG^{-1}\in\,^t\K_r$
(independent of the choice of $Y$).
We know there are $p^{r(r+1)/2}$ choices for $Y$, and
by Lemma 6.2, $\bbeta(n,r)$ choices for $G$.
also, with $n-s=r$ and $E\in\,^t\K_r\,^tG$, we have
$(M\ N)\equiv \left(\frac{1}{p}X_rM'X_r\ \,I\right)\ (\stufe).$
So when 
$$\left(\frac{1}{p}X_rM'X_r\,\  I\right)\in GL_n(\Z)(M_{\sigma}\ I)\Gamma_0(\stufe),$$
we get a contribution of 
$$\overline\chi(\overline pX_rM_{\sigma'}X_r,I)\chi(p^{n-r})p^{k(n-r)+(r-n)(r+n+1)/2}\bbeta(n,r)$$
toward $c_{\sigma,\sigma'}$ (and a contribution of 0 otherwise).

To determine when $\left(\frac{1}{p}X_rM'X_r\,\ I\right)\in GL_n(\Z)(M_{\sigma}\ I)\Gamma_0(\stufe)$,
take $E'\in SL_n(\Z)$, $\gamma\in \Gamma_0(\stufe)$; then take $E''\in SL_n(\Z)$, $\gamma'\in\Gamma_0(\stufe)$
so that $E''\equiv I\ (p)$, $E''\equiv E'\ (\stufe)$, $\gamma'\equiv I\ (p)$, $\gamma'\equiv\gamma\ (\stufe).$
Then set $E_r'=X_r^{-1}E''X_r$,
$$\gamma_r=\begin{pmatrix}\frac{1}{p}X_r\\&X_r^{-1}\end{pmatrix} \gamma' \begin{pmatrix} pX_r^{-1}\\&X_r\end{pmatrix};$$
so $E_r'\in SL_n(\Z)$, $\gamma_r\in\Gamma_0(\stufe)$.
Then $\left(\frac{1}{p}X_rM'X_r\ I\right)$ is equal to
$E'(M_{\sigma}\ I)\gamma$ or to $E'((1,-1)\cdot M_{\sigma}\ I)\gamma$
if and only if 
$$(M'\ I)\equiv E_r'X_r^{-1}(M_{\sigma}\ I)\begin{pmatrix} pX_r^{-1}\\&X_r\end{pmatrix}\gamma_r\ (\stufe)$$
or
$$(M'\ I)\equiv E_r'G_{\pm}X_r^{-1}(M_{\sigma}\ I)\begin{pmatrix} pX_r^{-1}\\&X_r\end{pmatrix}\gamma_{\pm}\gamma_r\ (\stufe).$$
Hence using Proposition 3.3, we have
$\left(\frac{1}{p}X_rM'X_r\ I\right)\in GL_n(\Z)(M_{\sigma}\ I)\Gamma_0(\stufe)$ if and only if
$(M'\ I)\in GL_n(\Z)(pX_r^{-1}M_{\sigma}X_r^{-1}\,\ I)\Gamma_0(\stufe).$  Note that when $r>0$, we can find
$G_r\in SL_n(\Z)$ so that $G_r\equiv\begin{pmatrix} \overline p^{r-1}\\&pI_{r-1}\\&&I\end{pmatrix}\ (\stufe),$
and then
$$G_r(pX_r^{-1}M_{\sigma}X_r^{-1}\ I)\begin{pmatrix} ^tG_r\\&G_r^{-1}\end{pmatrix}
\equiv ((p,\overline p^r)\cdot M_{\sigma}\ I)\ (\stufe).$$
Thus $\left(\frac{1}{p}X_rM'X_r\ I\right)\in GL_n(\Z)(M_{\sigma}\ I)\Gamma_0(\stufe)$ if and only if
$$(M'\ I)\in GL_n(\Z)((p,\overline p^r)\cdot M_{\sigma}\ I)\Gamma_0(\stufe).$$
Also, by Proposition 3.2, we have $\overline\chi(\gamma_r) \E_{\sigma'}=\E_{(p,\overline p^r)\cdot\sigma}.$
Therefore $\E_{\sigma}|T(p)=\sum_{r=0}^n \chi(p^{n-r)} p^{k(n-r)+(r-n)(r+n+1)/2}\E_{(p,\overline p^r)\cdot \sigma},$
as claimed.
\end{proof}

\smallskip\noindent
{\bf Definition.}  Let $\U_{\stufe}=(\Z/\stufe\Z)^{\times}$, $\psi\in\widehat{\U_{\stufe}\times\U_{\stufe}}$, the
character group of $\U_{\stufe}\times\U_{\stufe}$.  Set
$$\E_{\sigma,\psi}=\sum_{v,w\in\U_{\stufe}}\overline\psi(v,w) \E_{(v,w)\cdot\sigma}.$$
\smallskip

Below we will show that when non-zero, $\E_{\sigma,\psi}$ is an eigenform for all $T(p), T_j(p^2)$, $p$ prime not dividing $\stufe$.
Note that by orthogonality of characters,
$$\spn\{\E_{u\cdot\sigma}:\ u\in\U_{\stufe}\times\U_{\stufe}\ \}
=\spn\{\E_{\sigma,\psi}:\ \psi\in\widehat{\U_{\stufe}\times\U_{\stufe}}\ \}.$$
Also, we have $\psi(v,w)=\psi_1(v)\psi_2(w)$ where $\psi_1,\psi_2$ are characters on $\U_{\stufe}$; using Proposition 5.1,
when $\stufe$ is square-free we have 
$$\E_{\sigma,\psi}=\sum_{v,w\in\U_{\stufe}} \overline\psi(v,w)\left(\prod_{0<d\le n}\overline \chi_{\stufe_d}(v^d)\right)
\overline\chi_{\stufe_n}(w^2)\,\E_{\sigma},$$
so $\E_{\sigma,\psi}=0$ unless $\E_{\sigma}\not=0$ and $\psi_1=\prod_{0<d\le n}\overline\chi_{\stufe_d}^d$,
$\psi_2=\overline\chi_{\stufe_n}^2$ where $\stufe_d$ is the product of all primes $q|\stufe$ so that $\rank_qM_{\sigma}=d.$

\begin{cor}\label{Corollary 5.3} 
Suppose $\E_{\sigma}\not=0$, and let $p$ be a prime with $p\nmid\stufe$.
Let $\psi$ be a character on  $\U_{\stufe}\times\U_{\stufe}$; so $\psi(v,w)=\psi_1(v)\psi_2(w)$ where $\psi_1,\psi_2$ are characters on $\U_{\stufe}$.
Then 
$\E_{\sigma,\psi}|T(p)=\lambda_{\sigma,\psi}(p)   \E_{\sigma,\psi}$
where
$$\lambda_{\sigma,\psi}(p)
=\psi_1(p)\overline\psi_2(p^n)\cdot\prod_{i=1}^n (\psi_2\chi(p)p^{k-i}+1).$$
When $\stufe$ square-free, $\E_{\sigma}|T(p)=\lambda_{\sigma}(p)\E_{\sigma}$ and
$\widetilde\E_{\sigma}|T(p)=\lambda_{\sigma}(p)\widetilde\E_{\sigma}$
where $\widetilde \E_{\sigma}$ is as in Corollary 4.3 and
$$\lambda_{\sigma}(p)=\left(\prod_{0<d\le n}\chi_{\stufe_d}(p^d)\right)\,\prod_{i=1}^n\left(\chi(p)\overline\chi_{\stufe_n}(p^2)p^{k-i}+1\right).$$

\end{cor}

\begin{proof}  Write $\bbeta(m,r)$ for $\bbeta_p(m,r)$.
As in the proof of Theorem 5.2, we can assume $p^3|M_{\sigma}$.  Identify $p^{-1}$ with
$\overline p$ where $p\overline p\equiv 1\ (\stufe).$

With $v,w$ varying over $\U_{\stufe}$ and $r$ varying so that $0\le r\le n$,  we have
\begin{align*}
\E_{\sigma,\psi}|T(p)
&=p^{kn-n(n+1)/2}\sum_{v,w,r} \overline\psi(v,w)\chi(p^{n-r})p^{-kr+r(r+1)/2}\bbeta(n,r)\,\E_{(pv,\overline p^rw)\cdot\sigma}.
\end{align*}
Making the change of variables $v\mapsto \overline pv$ and $w\mapsto p^r w$, we get
$$
\E_{\sigma,\psi}|T(p)
=\psi_1(p)\chi(p^n) p^{kn-n(n+1)/2} S(n,k)\,\E_{\sigma,\psi}$$
where
$$S(n,k)=\sum_{r=0}^n \psi_2\chi(\overline p^r)p^{-kr+r(r+1)/2}\bbeta(n,r).$$
Using that $\bbeta(n,r)=p^r\bbeta(n-1,r)+\bbeta(n-1,r-1)$, we find that
\begin{align*}
S(n,k)&=(\psi_2\chi(\overline p)p^{1-k}+1) S(n-1,k-1)\\
&=\prod_{i=1}^n(\psi_2\chi(\overline p)p^{i-k}+1)\\
&=\psi_2\chi(\overline p^n)p^{-nk+n(n+1)/2}\prod_{i=1}^n(\psi_2\chi(p)p^{k-i}+1).
\end{align*}

Now suppose $\stufe$ is square-free.  With $\psi_1=\prod_{0<d\le n}\overline\chi_{\stufe_d}^d$ and 
$\psi_2=\overline\chi_{\stufe_n}^2$, we have $\E_{\sigma,\psi}=|\U_{\stufe}|^2\cdot\E_{\sigma}$;
recalling that $\chi_{\stufe_d}=\overline\chi_{\stufe_d}$ for $0<d<n$ (see Proposition 3.6),
the above result gives us $\E_{\sigma}|T(p)=\lambda_{\sigma}(p)\E_{\sigma}$, as claimed.  We also have
$\widetilde\E_{\sigma}=\sum_{\alpha\ge\sigma\,(\stufe)}a_{\sigma,\alpha}\E_{\alpha}$ with
$a_{\sigma,\sigma}=1$; since the Hecke operators commute, the multiplicity-one result of Corollary 4.3
tells us $\widetilde\E_{\sigma}=\lambda_{\sigma}(p)\widetilde\E_{\sigma}.$
\end{proof}

\begin{thm}\label{Theorem 5.4}  With $p$ a prime not dividing $\stufe$,
\begin{align*}
\E_{\sigma}|T_j(p^2) &= \bbeta_p(n,j)
\sum_{r+s\le j}\chi(p^{j-r+s})p^{k(j-r+s)-(j-r)(n+1)}\\
&\ \cdot\bbeta_p(j,r)\bbeta_p(j-r,s)\sym_p(j-r-s)
\E_{(1,p^{s-r})\cdot\sigma}
\end{align*}
(where $\sym_p(t)$ is the number of invertible, symmetric $t\times t$ matrices modulo $p$).
\end{thm}

\begin{proof}
To a large extent, we follow the line of reasoning in the proof of Theorem 5.2.
We write $\K_{r,s}$ for $\K_{r,s}(p)$, $X_{r,s}$ for $X_{r,s}(p)$, $\bbeta(m,r)$ for $\bbeta_p(m,r).$

As discussed at the beginning of the proof of Theorem 5.2, we can modify our representatives
$(M_{\sigma'}\ I)$ to assume $p^3|M_{\sigma'}$.  By Proposition 2.1, we have
\begin{align*}
\E_{\sigma}(\tau)|T_j(p^2)
&=\sum \chi(p^{j-n_0+n_2})\overline\chi(M,N) p^{j(k-n-1)}\\
&\ \cdot\det(MX_{n_0,n_2}^{-1}G^{-1}\tau+MX_{n_0,n_2}^{-1}Y\,^tG+NX_{n_0,n_2}\,^tG)^{-k}
\end{align*}
where $SL_n(\Z)(M\ N)$ varies over $SL_n(\Z)(M_{\sigma}\ I)\Gamma_0(\stufe)$,
$n_0,n_2\in\Z_{\ge0}$ vary subject to $n_0+n_2\le j$, $G\in SL_n(\Z)/\K_{n_0,n_2}$,
$Y\in\Y_{n_0,n_2}(p^2)$; 
note that we can assume that $Y\equiv 0\ (\stufe)$ and,
using Lemma 6.1, that $G\equiv I\ (\stufe)$.
Given $(M\ N)\in SL_n(\Z)(M_{\sigma}\ I)\Gamma_0(\stufe)$ and
$$Y=\begin{pmatrix} Y_0&Y_2&Y_3&0\\^tY_2&Y_1/p&0\\ ^tY_3&0\\ 0 \end{pmatrix}\in\Y_{n_0,n_2}(p^2),$$
we decompose $M,N$ into $3\times 4$ block matrices as follows.  First,
write $M=(M_9'\ M_{10}'\ M_{11}'\ M_{12}')$, $N=(N_9'\ N_{10}'\ N_{11}'\ N_{12}')$ where
$M_9', N_9'$ are $n\times n_0$, $M_{10}', N_{10}'$ are $n\times(j-n_0-n_2)$, and $M_{12}', N_{12}'$ are $n\times n_2$.
Let $s=\rank_p(M_9'\ M_{10}'\ N_{12}')$; using left multiplication from $SL_n(\Z)$, we can assume
$$M=\begin{pmatrix} pM_5'&pM_6'&M_7'&M_8'\\M_9&M_{10}&M_{11}&M_{12}\end{pmatrix},\ 
N=\begin{pmatrix} N_5'&N_6'&N_7'&pN_8'\\N_9&N_{10}&N_{11}&N_{12}\end{pmatrix},$$
where  $M_8',N_8'$ are $(n-s)\times n_2$, $M_9,N_9$ are $s\times n_0$, $M_{10},N_{10}$ are $s\times(j-n_0-n_2)$
(so $s=\rank_p(M_9\ M_{10}\ N_{12})$).  
Take $r$ so that
$$n-r=\rank_p\begin{pmatrix} M_5'&M_7'&N_6'+M_6'Y_1&N_7'&N_8'\\M_9&0&M_{10}Y_1&0&N_{12}\end{pmatrix}.$$
Thus using left multiplication from $SL_n(\Z)$ (leaving the lower $s$ rows fixed),
we can assume $p$ divides  the upper $r$ rows of $(M'_5\ M'_7\ N_6'+M_6'Y_1\ N'_7\ N'_8)$, and so
$$M=\begin{pmatrix} p^2M_1&pM_2&pM_3&M_4\\pM_5&pM_6&M_7&M_8\\M_9&M_{10}&M_{11}&M_{12}\end{pmatrix},\ 
N=\begin{pmatrix} N_1&N_2&pN_3&p^2N_4\\N_5&N_6&N_7&pN_8\\N_9&N_{10}&N_{11}&N_{12}\end{pmatrix}$$
with $M_1,N_1$ $r\times n_0$.  Also, since $p$ divides the upper $r$ rows of $N_6'+M_6'Y_1$,
we have $N_2\equiv -M_2Y_1\ (p)$ and
$$\rank_p\begin{pmatrix} M_5&M_7&N_6+M_6Y_1&N_7&N_8\\M_9&0&M_{10}Y_1&0&N_{12}\end{pmatrix}=n-r.$$
Note that we necessarily have $\rank_p(M_4\ N_1\ N_2)=r$; since $Y_1$ is invertible modulo $p$
and $N_2\equiv -M_2Y_1\ (p)$, we have $\rank_p(M_2\ M_4\ N_1)=r$.
Set
$$(M'G\ N'\,^tG^{-1})=
X_{r,s}^{-1}(MX_{n_0,n_2}^{-1}\ NX_{n_0,n_2}+MX_{n_0,n_2}^{-1}Y).$$
Hence $M',N'$ are integral, and with $Y'=\begin{pmatrix} Y_0&Y_2&Y_3&0\\ ^tY_2\\^tY_3\\0\end{pmatrix},$
we have
\begin{align*}
&\rank_p(M'\ N')\\
&\quad=\rank_p(M'G\ N'\,^tG^{-1})\\
&\quad=\rank_p(M'G\ \,N'\,^tG^{-1} - M'GY')\\
&\quad=\rank_p{\begin{pmatrix} M_1&M_2&M_3&M_4&N_1&(N_2+M_2Y_1)/p&N_3&N_4\\
M_5&0&M_7&0&0&N_6+M_6Y_1&N_7&N_8\\
M_9&0&0&0&0&M_{10}Y_1&0&N_{12}\end{pmatrix}}\\
&\quad\ge\rank_p{\begin{pmatrix} M_1&M_2&0&M_4&N_1&0&0&0\\
M_5&0&M_7&0&0&N_6+M_6Y_1&N_7&N_8\\
0&0&0&0&0&M_{10}Y_1&0&N_{12}\end{pmatrix}} \\
&\quad=n.
\end{align*}
So $(M'\ N')$ is an integral coprime pair, and
\begin{align*}
&\det(MX_{n_0,n_2}^{-1}G^{-1}\tau+NX_{n_0,n_2}\,^tG+MX_{n_0,n_2}^{-1}Y\,^tG)^{-k}\\
&\quad=p^{k(s-r)}\det(M'\tau+N')^{-k}.
\end{align*}

Now take an index $\sigma'$, $E\in\ ^t\K_{r,s}\backslash SL_n(\Z)$, and set
$$(M\ N)=X_{r,s}E(M_{\sigma'}GX_{n_0,n_2}\ \,^tG^{-1}X_{n_0,n_2}^{-1}-M'GYX_{n_0,n_2}^{-1})$$
where $M'$ is $M_{\sigma'}$ or $(1,-1)\cdot M_{\sigma'}.$
We first determine exactly when $(M\ N)$ is an integral coprime pair, and then we determine when
$$(M\ N)\in GL_n(\Z)(M_{\sigma}\ I)\Gamma_0(\stufe).$$

Recall that $G=G_1G_2$ (as described in Proposition 2.1) with $G_1$ varying over $SL_n(\Z)/\K_{n_0,n_2}$ and
$G_2=\begin{pmatrix} I_{n_0}\\&G'\\&&I_{n_2}\end{pmatrix}$
with $G'$ varying over $SL_{n'}(\Z)/\,^t\K'_{j'}$ where $n'=n-n_0-n_2$, $j'=j-n_0-n_2$; 
also
recall that since we can assume $p^3|M_{\sigma'}$,
we have $M\equiv 0\ (p)$.  
So to have $(M\ N)$ integral and coprime, we need $X_{r,s}E\,^tG^{-1}X_{n_0,n_2}$
integral and invertible modulo $p$.  
Since $X_{n_0,n_2}$ and $G_2$ commute,
to have $X_{r,s}E\,^tG^{-1}X_{n_0,n_2}$ integral, we need
 $E\,^tG_1^{-1}=\begin{pmatrix} N_1&N_2&N_3\\pN_4&N_5&N_6\\p^2N_7&pN_8&N_9\end{pmatrix}$
where $N_1$ is $r\times n_0$, $N_9$ is $s\times n_2$, which means we have $\rank_pN_1=n_0$,
$\rank_pN_9=s$.
Then
$$N\equiv\begin{pmatrix} N_1&pN_2&p^2N_3\\N_4&N_5&pN_6\\N_7&N_8&N_9\end{pmatrix}\ (p^2),$$
so to have $N$ invertible modulo $p$,
we need
$\rank_pN_1=r$, $\rank_pN_9=n_2$, meaning $r=n_0$, $s=n_2$; we then must have $\rank_pN_5=n-n_0-n_2$
since $E\,^tG_1^{-1}$ is invertible modulo $p$.
So suppose $r=n_0$, $s=n_2$, and fix $G_1$.  
Then we have
$X_{r,s}E\,^tG_1^{-1}X_{r,s}^{-1}$ integral if and only if $E\,^tG_1^{-1}\in\,^t\K_{r,s}$;
consequently, $(M\ N)$ is integral and coprime if and only if $r=n_0$, $s=n_2$, and
$E\in\,^t\K_{r,s}\,^tG_1$.

To summarise:  For any choices of $G_1\in SL_n(\Z)/\K_{n_0,n_2}$, $G_2\in SL_{n'}(\Z)/\K'_{j'}$, 
$Y\in \Y_{r,s}(p^2)$, we have 
$$(M\ N)=X_{r,s}E(M_{\sigma'}GX_{n_0,n_2}\ \,^tG^{-1}X_{n_0,n_2}^{-1}-M'GYX_{n_0,n_2}^{-1})$$
 integral and coprime if and only if 
$r=n_0$, $s=n_2$, and $E\in \,^t\K_{r,s}\,^tG_1.$
There are $p^{rs}\bbeta(n,r)\bbeta(n-r,s)$ choices for
 $G_1$, $\bbeta(n-r-s,j-r-s)$ choices for $G_2$, and $p^{r(r+1)+r(n-r-s)}\sym_p(j-r-s)=p^{r(n-s+1)}\sym_p(j-r-s)$ choices
for $Y$.

With $(M\ N)$ integral and coprime, we have $\overline \chi (M,N)=\overline \chi (X_{r,s}M_{\sigma'}X_{r,s},I)$,
and arguing as in the proof of Theorem 5.2, we find that
$$(M\ N)\in GL_n(\Z)(M_{\sigma}\ I)\Gamma_0(\stufe)$$ if and only if
$(M_{\sigma'}\ I)\in GL_n(\Z)((1,p^{r-s})\cdot M_{\sigma}\ I)\Gamma_0(\stufe).$
Also, $\E_{(1,p^{r-s})\cdot\sigma}=\overline\chi((1,p^{r-s})\cdot M_{\sigma},I) \E_{\sigma'}.$
Note also that
\begin{align*}
&\bbeta(n,r)\bbeta(n-r,s)\bbeta(n-r-s,j-r-s)\\
&\quad=\frac{\bmu(n,r)\bmu(n-r,s)\bmu(n-r-s,j-r-s)}{\bmu(r,r)\bmu(s,s)\bmu(j-r-s,j-r-s}
\frac{\bmu(j,r+s)}{\bmu(j,r+s)}\\
&\quad=\bbeta(n,j)\bbeta(j,r)\bbeta(j-r,s).
\end{align*}
For this the theorem follows.
\end{proof}

We now choose a different set of generators for the Hecke algebra to obtain more attractive eigenvalues.

\smallskip\noindent
{\bf Definitions.}  Let $p$ be a prime not dividing $\stufe$, and fix $j$, $1\le j\le n$.
As in \cite{HW}, we set
$$\widetilde T_j(p^2)=\sum_{0\le\ell\le j}\chi(p^{j-\ell})p^{(j-\ell)(k-n-1)}\bbeta_p(n-\ell,j-\ell) T_{\ell}(p^2)$$
where $T_0(p^2)$ is the identity map.
The effect of this averaging is to remove on $Y_1$ the condition that $p\nmid\det Y_1$
(where $Y\in\Y_{n_0,n_2}(p^2)$ is as described in Proposition 2.1).
For $u\in\U_{\stufe}$ and 
$\gamma=\begin{pmatrix} I&0\\M&I\end{pmatrix}\in Sp_n(\Z)$, we define
$$R(u)\E_{\gamma}=\E_{(1,u)\cdot\gamma},$$
and we extend $R(u)$ linearly to 
$\Eis_k^{(n)}(\stufe,\chi)$,
which we know is spanned by all such $\E_{\gamma};$ by Proposition 5.1, $R(u)$ is well-defined.
By Theorems 5.2 and 5.4, we see that $R(u)$ commutes with $T(p)$ and $T_j(p^2)$ ($p$ prime, $p\nmid\stufe$, $1\le j\le n$).
Thus
$$\{T(p), T_j(p^2), R(p):\ \text{prime } p\nmid\stufe,\ 0\le j\le n\ \}$$
generates a commutative algebra of operators on $\Eis_k^{(n)}(\stufe,\chi).$
Set
$$T'_j(p^2)=\sum_{i=0}^j(-1)^i p^{i(i-1)/2}\bbeta_p(n-j+i,i)\widetilde T_{j-i}(p^2)R(\overline p^i)$$
where $p\overline p\equiv 1\ (\stufe)$.
(Recall that $R(u)$ is defined for any $u\in\Z$ with $(u,\stufe)=1$, so $R(\overline p^i)$ makes sense.)

\begin{cor}\label{Corollary 5.5} 
We have
$$\E_{\sigma,\psi}|T'_j(p^2)=\lambda'_{j;\sigma,\psi}(p^2)\,\E_{\sigma,\psi}$$
where
$$\lambda'_{j;\sigma,\psi}(p^2)=\bbeta_p(n,j)p^{(k-n)j+j(j-1)/2}\chi(p^j)\prod_{i=1}^j(\psi_2\chi(p)p^{k-i}+1).$$
When $\stufe$ is $square-free$, we have $\E_{\sigma}|T'_j(p^2)=\lambda'_{j;\sigma}(p^2) \E_{\sigma}$
and $\widetilde\E_{\sigma}|T'_j(p^2)=\lambda'_{j;\sigma}(p^2) \widetilde\E_{\sigma}$ where
$$\lambda'_{j;\sigma}(p^2)=\bbeta_p(n,j)p^{(k-n)j+j(j-1)/2}\chi(p^j)\prod_{i=1}^j(\chi\overline\chi_{\stufe_n}^2(p)p^{k-i}+1).$$
\end{cor}

\begin{proof}  Write $\bbeta(m,r)$ for $\bbeta_p(m,r)$.
Using Theorem 5.4, averaging over $v,w\in\U_{\stufe}$ (and replacing $w$ by $wp^{r-s}$ inside the sum on $r,s$
in the formula of Theorem 5.4), we get
$\E_{\sigma,\psi}|\widetilde T_j(p^2)=\widetilde\lambda_{j;\sigma,\psi}(p^2) \E_{\sigma,\psi}$ where
\begin{align*}
\widetilde\lambda_{j;\sigma,\psi}(p^2)&= \sum_{\ell,r,s}\chi(p^{j+s-r})\psi_2(p^{s-r})p^{j(k-n-1)+r(n+1)+k(s-r)}\\
&\quad \cdot \bbeta(n,\ell)\bbeta(n-\ell,j-\ell)\bbeta(\ell,r)\bbeta(\ell-r,s) \sym_p(\ell-r-s)
\end{align*}
where $0\le \ell\le j$, $0\le r+s\le \ell$, or equivalently, $0\le r+s\le j$, $r+s\le \ell\le j$.
We make the change of variables $\ell\mapsto j-\ell$ and use that
\begin{align*}
&\bbeta(n,j-\ell)\bbeta(n-j+\ell,\ell)\bbeta(j-\ell,r)\bbeta(j-\ell-r,s)\frac{\bmu(j,\ell)}{\bmu(j,\ell)}\\
&\quad = \bbeta(n,j)\bbeta(j,r)\bbeta(j-r,s)\bbeta(j-r-s,\ell)\\
&\quad = \bbeta(n,j)\bbeta(j,r)\bbeta(j-r,s)\bbeta(j-r-s,j-\ell-r-s).
\end{align*}
Now we make the change of variable $\ell\mapsto j-\ell-r-s$, we get
\begin{align*}
\widetilde\lambda_{j;\sigma,\psi}(p^2)&=\bbeta(n,j)\sum_{0\le r+s\le j}\chi(p^{j+s-r})\psi_2(p^{s-r})p^{j(k-n-1)+r(n+1)+k(s-r)}\\
&\quad\bbeta(j,r)\bbeta(j-r,s)\cdot \sum_{0\le \ell\le j-r-s} \bbeta(j-r-s,\ell)\sym_p(\ell).
\end{align*}
By Lemma 6.7, the sum on $\ell$ is $p^{(j-r-s)(j-r-s+1)/2}.$

We have $\E_{\sigma,\psi}|R(\overline p^i)=\overline\psi_2(p^i) \E_{\sigma,\psi}.$
Thus $\E_{\sigma,\psi}|T'_j(p^2)=\lambda'_{j;\sigma,\psi}(p^2) \E_{\sigma,\psi}$ where
\begin{align*} 
&\lambda'_{j;\sigma,\psi}(p^2)\\
&\quad= \sum_{i,r,s}(-1)^i p^{i(i-1)/2}\bbeta(n-j+i,i) \overline\psi_2(p^i) \chi(p^{j-i+s-r})\psi_2(p^{s-r})\\
&\qquad \cdot p^{(j-i-r)(k-n-1)+ks+(j-i-r-s)(j-i-r-s+1)/2}\\
&\qquad \bbeta(n,j-i)\bbeta(j-i,r)\bbeta(j-i-r,s)
\end{align*}
where $0\le i\le j$, $0\le r\le j-i$, $0\le s\le j-i-r$.
Making the change of variable $r\mapsto j-i-r$ we get
\begin{align*} \lambda'_{j;\sigma,\psi}(p^2)
&= \sum_{i,r,s}(-1)^i p^{i(i-1)/2}\chi(p^{r+s})\psi_2(p^{s+r-j}) p^{r(k-n-1)+ks+(r-s)(r-s+1)/2}\\
&\quad \cdot \bbeta(n-j+i,i)\bbeta(n,j-i)\bbeta(j-i,j-i-r)\bbeta(r,s),
\end{align*}
where $0\le i\le j$, $0\le r\le j-i$, $0\le s\le r$, or equivalently, $0\le r\le j$, $0\le i\le j-r$, $0\le s\le r$.
Note that
\begin{align*}
&\bbeta(n-j+i,i)\bbeta(n,j-i)\bbeta(j-i,r)\frac{\bmu(j,i)}{\bmu(j,i)}\\
&\quad = \bbeta(n,j)\bbeta(j,r)\bbeta(j-r,i).
\end{align*}
Also, using the relation $\bbeta(m,r)=p^r\bbeta(m-1,r)+\bbeta(m-1,r-1)$, we get
$$\sum_{i=0}^{j-r}(-1)^i p^{i(i-1)/2}\bbeta(j-r,i) = \begin{cases} 1&\text{if $j=r$,}\\ 0 &\text{otherwise.}\end{cases}$$
Hence $\lambda'_{j;\sigma,\psi}(p^2)=\chi(p^j)p^{j(k-n-1)+j(j+1)/2}\bbeta(n,j) S(j,k-j)$ where
$$S(j,y)=\sum_{s=0}^j p^{ys+s(s-1)/2} \chi\psi_1(p^s)\bbeta(j,s).$$
Using the identity $\bbeta(m,s)=p^s\bbeta(m-1,s)+\bbeta(m-1,s-1)$, we have
\begin{align*}
S(j,k-j)&=(\chi\psi_2(p)p^{k-j}+1)S(j-1,k-j+1)\\
&=\prod_{i=1}^j(\chi\psi_2(p)p^{k-i}+1),
\end{align*}
proving the corollary.
\end{proof}

\bigskip
\section{Lemmas}
\smallskip

\begin{lem}\label{Lemma 6.1}  Suppose $\stufe',\stufe''\in\Z_+$ with $(\stufe',\stufe'')=1$.
\item{(a)}  Given any $E'\in SL_n(\Z)$, there is some $E\in SL_n(\Z)$ so that $E\equiv E'\ (\stufe')$ and
$E\equiv I\ (\stufe'')$.  
\item{(b)}  Suppose $\gamma\in\Gamma_0(\stufe')$.  Then there is some $\gamma'\in\Gamma_0(\stufe')$
so that $\gamma'\equiv\gamma\ (\stufe')$ and $\gamma'\equiv I\ (\stufe'').$
\end{lem}

\begin{proof}  
(a)  Choose $y,z\in\Z$ so that $y\stufe''+z\stufe'=1$.  Thus $(y\stufe'',z\stufe')=1$, so 
there are $w,x\in\Z$ so that $wz(\stufe')^2-xy(\stufe'')^2=1$.  Hence
$G_0=\begin{pmatrix} w\stufe'&x\stufe''\\y\stufe''&z\stufe'\end{pmatrix}\in SL_2(\Z)$
with $G_0\equiv I\ (\stufe'')$, $G_0\equiv\begin{pmatrix} 0&-1\\1&0\end{pmatrix}\ (\stufe')$.  
For $b\in\Z$, $G_1=\begin{pmatrix} 1&by\stufe''\\0&1\end{pmatrix}$, 
$G_2=\begin{pmatrix} 1&0\\by\stufe''&1\end{pmatrix}$, 
we have $G_1,G_2\in SL_2(\Z)$ with $G_1,G_2\equiv I\ (\stufe'')$,
$G_1\equiv \begin{pmatrix} 1&b\\0&1\end{pmatrix}\ (\stufe')$, $G_2\equiv\begin{pmatrix} 1&0\\b&1\end{pmatrix}\ (\stufe')$.
Also, for $a\in\Z$ so that $q\nmid a$,  take $c\in\Z$ so that $c\equiv 1\ (\stufe'')$, $c\equiv a\ (\stufe')$.
Thus $(c,\stufe'\stufe'')=1$ so there are $u,v\in\Z$ so that $cu-(\stufe'\stufe'')^2v=1$.  Set $G_3=\begin{pmatrix} c&\stufe'\stufe''\\v\stufe'\stufe''&u\end{pmatrix}.$
So $G_3\in SL_2(\Z)$, $G_3\equiv I\ (\stufe'')$, $G_3\equiv\begin{pmatrix} c&0\\0&\overline c\end{pmatrix}\ (\stufe')$
(where $c\overline c\equiv 1\ (\stufe')$).   For any $d$, $0\le d\le n-2$, the map
$G\mapsto \begin{pmatrix} I_d\\&G\\&&I_{n-d-2}\end{pmatrix}$ is an embedding of $SL_2(\Z)$ into $SL_n(\Z)$.
Thus we have matrices that allow us to perform "local" elementary row and column operations modulo $\stufe'$
within $SL_n(\Z)$.  Hence, given some $E'\in SL_n(\Z)$, there are $G,G'\in SL_n(\Z)$ so that
$G,G'\equiv I\ (\stufe'')$ and $GE'G'\equiv I\ (\stufe')$; hence with $E=(G'G)^{-1}$, we have
$E\in SL_n(\Z)$ with $E\equiv I\ (\stufe'')$, $E\equiv E'\ (\stufe').$

(b) Write $\gamma=\begin{pmatrix} A&B\\C&D\end{pmatrix}$; so $C\equiv0\ (\stufe')$.  Set $a=\det A$; since
$A\,^tD\equiv I\ (\stufe')$, we can choose $\overline a\in\Z$ so that $a\overline a\equiv 1\ (\stufe').$
So we can choose $G\in SL_n(\Z)$ so that $G\equiv\begin{pmatrix}\overline a\\&I\end{pmatrix} A\ (\stufe'),$
$G\equiv I\ (\stufe'').$  We know $A\,^tB$ is symmetric; set 
$W=\stufe''\overline\stufe'' G^{-1}\begin{pmatrix}\overline a\\&I\end{pmatrix} B$
where $\stufe''\overline\stufe''\equiv 1\ (\stufe').$  Now take
$\begin{pmatrix} w&x\\y&z\end{pmatrix}\in SL_2(\Z)$ so that
$\begin{pmatrix} w&x\\y&z\end{pmatrix}\equiv\begin{pmatrix} a\\&\overline a\end{pmatrix}\ (\stufe')$ and
$\begin{pmatrix} w&x\\y&z\end{pmatrix}\equiv I\ (\stufe'').$  Set
$$\gamma'=\begin{pmatrix} w&&x\\&I_{n-1}&&0\\y&&z\\&0&&I_{n-1}\end{pmatrix}
\begin{pmatrix} G&GW\\0&^tG^{-1}\end{pmatrix};$$
so $\gamma'\in\Gamma_0(\stufe')$, $\gamma\equiv\gamma\ (\stufe'),$ and $\gamma'\equiv I\ (\stufe'').$
\end{proof}

\begin{lem}\label{Lemma 6.2} Let $\Lambda=\Z x_1\oplus\cdots\oplus\Z x_n$;
fix a prime $q$ and let $\K_d=\K_d(q)$.  The elements of
$SL_n(\Z)/\K_d$ are in one-to-one correspondence with lattices $\Omega$ where
$q\Lambda\subseteq\Omega\subseteq\Lambda$ and $[\Lambda:\Omega]=q^d$.
The correspondence is given as follows: For $G\K_d\in SL_n(\Z)/\K_d$, $\Omega$
is the lattice with basis
$$(x_1\ \ldots\ x_n)G\begin{pmatrix} qI_d\\&I\end{pmatrix}.$$
Further, the number of such $\Omega$ is $\bbeta(n,r).$
\end{lem}

\begin{proof} 
Given $G\in SL_n(\Z)$, we map $G$ to the sublattice of $\Lambda$ with basis
$$(x_1\ \ldots\ x_n)G\begin{pmatrix} qI_d\\&I\end{pmatrix}.$$
Clearly each $\Omega$ described in the lemma can be obtained this way.
Further, for $H\in SL_n(\Z)$,
$$(x_1\ \ldots\ x_n)GH\begin{pmatrix} qI_d\\&I\end{pmatrix}$$
is also a basis for $\Omega$ if and only if $H\in\K_{d}.$
Also, each such $\Omega$ corresponds to a dimension $n-r$ subspace of $\Lambda/q\Lambda$,
of which there are $\bbeta(n,n-r)=\bbeta(n,r).$
\end{proof}

For $d,r\ge0$, $d+r\le n$, let
$\K_{d,r}(q)$ be the subspace of $SL_n(\Z)$ consisting of matrices
$$\begin{pmatrix} G_1&B_1&B_2\\C_1&G_2&B_3\\C_2&C_3&G_3\end{pmatrix}$$
where $G_1$ is $d\times d$, $G_3$ is $r\times r$, $B_1,B_3\equiv 0\ (q)$,
$B_2\equiv0\ (q^2).$

\begin{lem}\label{Lemma 6.3}  
Let $\Lambda=\Z x_1\oplus\cdots\oplus\Z x_n$; fix a prime $q$ and let
$\K_{d,r}=\K_{d,r}(q)$.
For $\Omega$ a sublattice of $\Lambda$ containing $q^2\Lambda$,
let $m_i$ denote the multiplicity of $q^i$ among the invariant factors $\{\Lambda:\Omega\}$.
Then the elements of $SL_n(\Z)/\K_{d,r}$ are in one-to-one
correspondence with sublattices $\Omega$ of $\Lambda$ containing $q^2\Lambda$
with $m_0=r$ and $m_2=d$.  
The correspondence is given as follows: For $G\in SL_n(\Z)/\K_{d,r}$, $\Omega$
is the lattice with basis
$$(x_1\ \ldots\ x_n)G\begin{pmatrix} q^2I_d\\&qI\\&&I_r\end{pmatrix}.$$
Further, there are $q^{dr}\bbeta(n,d)\bbeta(n-d,r)$ such $\Omega$.
\end{lem}

\begin{proof}  
Given $G\in SL_n(\Z)$, we map $G$ to the sublattice of $\Lambda$ with basis
$$(x_1\ \ldots\ x_n)G\begin{pmatrix} q^2I_d\\&qI\\&&I_{r}\end{pmatrix}.$$
Clearly each $\Omega$ described in the lemma can be obtained this way.
Further, for $H\in SL_n(\Z)$,
$$(x_1\ \ldots\ x_n)GH\begin{pmatrix} q^2I_d\\&qI\\&&I_{r}\end{pmatrix}$$
is also a basis for $\Omega$ if and only if $H\in\K_{d,r}.$

On the other hand, given such $\Omega$, we have $\Omega=q^2\Lambda_0\oplus q\Lambda_1\oplus \Lambda_2$
where $\Lambda=\Lambda_0\oplus\Lambda_1\oplus\Lambda_2$ with $\rank\Lambda_0=d$, $\rank\Lambda_2=r$.
We can construct all such $\Omega$ as follows.  First let $\Delta$ be the preimage in $\Lambda$ of
a dimension $n-d$ subspace of $\Lambda/q\Lambda$; there are $\bbeta(n,n-d)=\bbeta(n,d)$ such subspaces.
Then let $\Omega$ be the preimage in $\Delta$ of a dimension $r$ subspace of $\Delta/q\Delta$ that is
independent of $\overline{q\Lambda}$; there are $q^{dr}\bbeta(n-d,r)$ choices.
Since $\Delta=\frac{1}{q}\Omega\cap\Lambda$, a different choice in step 1 or step 2 of
this construction yields a different lattice $\Omega$.
\end{proof}

\noindent{\bf Remark.}
 Let $\Lambda=\Z x_1\oplus\cdots\oplus\Z x_n$, $\Lambda^{\#}=\Z y_1\oplus\cdots\oplus\Z y_n$
where $(y_1\ \ldots\ y_n)$ is the basis dual to $(x_1\ \ldots\ x_n)$.  Then for $G\in SL_n(\Z)$,
the basis dual to $(x_1\ \ldots\ x_n)G$ is $(y_1\ \ldots\ y_n)\,^tG^{-1};$ thus the elements of
$SL_n(\Z)/\K_d(q)$ are in one-to-one correspondence with subspaces $\F y_1'\oplus\cdots\oplus\F y_d'
\subseteq \Lambda^{\#}/q\Lambda^{\#}$ .  Similarly, the elements of $SL_n(\Z)/\K_{d,r}(q)$ are in
one-to-one correspondence with sublattices $\Omega'$ where 
$q^2\Lambda^{\#}\subseteq \Omega'\subseteq\Lambda^{\#}$ and
$$(y_1\ \ldots\ y_n)\,^tG^{-1}\begin{pmatrix} I_d\\&qI\\&&q^2I_r\end{pmatrix}$$
is a basis for $\Omega'$.

\begin{lem}\label{Lemma 6.4}  Fix a prime $q$;
suppose $M'\in\Z^{n,n}$ with
$d'=\rank_qM'$.  Let $\K_=\K_d(q)$, $\K_{m,r}=\K_{m,r}(q)$, $\bbeta(m,r)=\bbeta_q(m,r)$.
\item{(a)}  For $0\le d\le d'$, there are $q^{d(n-d')}\bbeta(d',d)$ choices for $E\in \K_d\backslash SL_n(\Z)$
so that the top $d$ rows of $EM'$ are linearly independent modulo $q$.
\item{(b)} For $r,m,s\ge 0$ so that $d'-r\le m+s\le d'$, there are
$$\bbeta(n-d',n-r-m-s)\bbeta(d',m+s)\bbeta(m+s,m)q^{m(n+r-d')+s(r+m+s-d')}$$
choices for $E\in \K_{m,r}\backslash SL_n(\Z)$ so that the $q$-rank of the top $m$
rows of $EM'$ is $m$ and the $q$-rank of the top $n-r$ rows of $EM'$ is $m+s$.
\end{lem}

\begin{proof} 
(a) Take $E_0\in SL_n(\Z)$ so that $q$ divides the lower $n-d'$ rows of $E_0M'$; as $E$ varies over a set of representatives
for $\K_d\backslash SL_n(\Z)$, so does $EE_0$.  Thus we may as well assume that $q$ divides the lower $n-d'$ 
rows of $M'$.
We know by Lemma 6.2 and the remark
preceding this lemma that each $E\in\K_d\backslash SL_n(\Z)$ corresponds to a sublattice
$\Omega=\F y_1\oplus\cdots\oplus\F y_n$ of $\Lambda=\F x_1\oplus\cdots\oplus\F x_n$ with $[\Lambda:\Omega]=q^{n-d}$,
where $$\begin{pmatrix} y_1\\ \vdots\\ y_n\end{pmatrix} = \begin{pmatrix} I_d\\ &qI_{n-d}\end{pmatrix} E\begin{pmatrix} x_1\\ \vdots\\ x_n\end{pmatrix}.$$
Thus $\rank_q\begin{pmatrix} I_d\\ &qI_{n-d}\end{pmatrix} EM'=d$ if and only if $E$ is chosen so that $\F y_1\oplus\cdots\oplus\F y_d$
is independent of $\F x_{d'+1}\oplus\cdots\oplus\F x_n$; there are $\bbeta(d',d)q^{d(n-d')}$ such subspaces.

(b)  Let $\Lambda=\Z x_1\oplus\cdots\oplus\Z x_n$, and let $\Lambda(\,^tM') \mod q$
denote the subspace of $\F^{n,1}$ obtained by replacing each $x_i$ by column $i$ of $^tM'$ modulo $q$.
We know that each element $E\in \K_{m,r}\backslash SL_n(\Z)$ corresponds to a lattice
$\Omega=\Lambda_0\oplus q\Lambda_1\oplus q^2\Lambda_2$ with basis
$$(x_1\ \ldots\ x_n)\,^tE\begin{pmatrix} I_{m}\\&qI\\&&q^2I_r\end{pmatrix}.$$
We want to choose $\Omega$ so that, with $\Delta=\frac{1}{q}\Omega\cap\Lambda=\Lambda_0\oplus\Lambda_1
\oplus q\Lambda_2$, the map $\Lambda\mapsto\Lambda(\,^tM')\mod q$ takes $\Omega$ to a dimension
$m$ subspace and $\Delta$ to a dimension $m+s$ subspace.  Let $m_i=\rank\Lambda_i$.

Given $\Omega$ a sublattice of $\Lambda$ containing $q^2\Lambda$ and with $m_0=m$, $m_2=r$,
$\Omega$ determines a unique dimension $n-r$
sublattice $\overline\Delta=\overline{\frac{1}{q}\Omega\cap\Lambda}$
of $\Lambda/q\Lambda$, and then with $\Delta$ the preimage of $\overline\Delta$ in $\Lambda$, and $\Omega$ determines
a unique dimension $m$ sublattice $\overline\Omega$ of $\Delta/q\Delta.$
Thus we can build all $\Omega$ corresponding to $\K_{m,r}\backslash SL_n(\Z)$ by first choosing
a dimension $n-r$ subspace $\overline\Delta$ of $\Lambda/q\Lambda$; then the preimage of 
$\overline\Delta$ is $\Delta=\Delta_1\oplus q\Lambda_2$ where
$\Lambda_2$ has rank $r$ and $\Lambda=\Delta_1\oplus\Lambda_2$.  Then in $\Delta/q\Delta$, we choose
a dimension $m$ subspace $\overline\Omega$ that is independent of $\overline{q\Lambda}
=\overline{q\Lambda_2}$; then the preimage of $\overline\Omega$ in $\Delta$ is $\Omega$.
So
each $\Omega$ corresponds to a (unique) dimension $n-r$ subspace $\overline\Delta$ of $\Lambda/q\Lambda$,
and a (unique) dimension $m$ subspace of $\Delta/q\Delta$.

We know $d'=\dim\Lambda(\,^tM') \mod q$, so $\Lambda=W\oplus R$ where
$$R=\ker\big(\Lambda\mapsto\Lambda(\,^tM') \mod q\big).$$  So in $\Lambda/q\Lambda$,
$\dim\overline R=n-d'$.  We choose $\overline\Delta$ of dimension $n-r$ in $\Lambda/q\Lambda$
so that $\dim\overline\Delta\cap\overline R=n-r-m-s.$  Thus there are
$$\bbeta(n-d',n-r-m-s)\bbeta(d',m+s)q^{(m+s)(m+s+r-d')}$$
choices for $\Delta$ so that $\Lambda\mapsto\Lambda(\,^tM') \mod q$ takes $\Delta$ to a
dimension $m+s$ subspace.  Then $\Delta/q\Delta=\overline U\oplus\overline R'$ where
$\dim \overline U=m+s$ and $R'\subseteq R+q\Lambda$.  We choose $\overline\Omega$
of dimension $m$ and independent of $\overline R'$; so we have
$\bbeta(m+s,m)q^{m(n-m-s)}$
choices for $\Omega\subseteq\Delta$ so that $\Lambda\mapsto\Lambda(\,^tM') \mod q$
takes $\Omega$ to a dimension $m$ subspace.
\end{proof}

\begin{lem}\label{Lemma 6.5} Fix a prime $q$ and write $\bbeta(m,r)$ for $\bbeta_q(m,r).$
Suppose $(M'\ N')$ is an $n\times n$ coprime symmetric pair such that
$$M'\equiv\begin{pmatrix} A_1&0&0&0\\0&C&0&0\\0&0&0&0\\0&0&C'&0\end{pmatrix}\ (q)$$
with $A_1$ $d_1\times d_1$, $C$ $(d_4+d_5)\times(d_4+d_5)$, $C'$ $(d_7+d_8)\times(d_7+d_8)$,
and $A_1, C, C'$ invertible modulo $q$.  (So with $d'=\rank_q M'$, we have
$d'=d_1+d_4+d_5+d_7+d_8$.)  Suppose also that $d_1+d_5+d_7\le j\le n-d_4-d_8,$
and set $r=j-d_1-d_5+d_8$.  Then there are
$$\bbeta(d_4+d_5,d_4)\bbeta(d_7+d_8,d_8)q^{(d_4+d_8)(j-d_1-d_5)-d_7d_8}$$
choices for $G\in SL_n(\Z)/\K_j$ so that, writing
$$M'G=\begin{pmatrix} M_1&M_2\\M_3&M_4\\M_5&M_6\end{pmatrix},\ 
N'\,^tG^{-1}=\begin{pmatrix} N_1&N_2\\N_3&N_4\\N_5&N_6\end{pmatrix}$$
with $M_1,N_1$ $d_1\times j$, $M_5,N_5$ $r\times j$, we have
$\rank_qM_1=d_1$,  $M_2\equiv 0\ (q),$
$\rank_qM_4=d_4$, $\rank_q\begin{pmatrix} M_4\\M_6\end{pmatrix}=d_4+d_8$,
$\rank_q\begin{pmatrix} M_1\\M_3\end{pmatrix}=d_1+d_5$,
the lower $n-r-d_1-d_4-d_5$ rows of $N_3$ are 0 modulo $q$, and the upper
$r-d_7-d_8$ rows of $N_5$ have $q$-rank $r-d_7-d_8$.
\end{lem}

\begin{proof}
Let $V=\F x_1\oplus\cdots\oplus\F x_n$.  We know by Lemma 6.2 that the elements 
$G\in SL_n(\Z)/\K_j$ are in one-to-one correspondence with the subspaces
$W=\F x_{j+1}'\oplus\cdots\oplus\F x_n'$ where
$(x_1'\ \ldots\ x_n')=(x_1\ \ldots\ x_n)\overline G$.  We translate the lemma's
criteria on $G$ to criteria on $W$, and then count such $W$.

Let $V^{\#}=\F y_1\oplus\cdots\oplus\F y_n$ be 
the dual space for $V$; so for $G\in SL_n(\F)$, $(y_1\ \ldots\ y_n)\,^tG^{-1}$
is the basis dual to $(x_1\ \ldots\ x_n)\overline G$.
Let $V(M')$ denote the subspace of $\F^{n,1}$
obtained by replacing each $x_i$ by the $i$th column of $M'$ modulo $q$.
We split $V$ as $V_1\oplus V_2\oplus V_3\oplus V_4$
as follows.  Let $(a_1\ \ldots\ a_n)$ denote (the columns of) the top $d_1+d_4+d_5$ rows of $M'$,
$(g_1\ \ldots\ g_n)$ the top $d_1$ rows of $M'$; set
\begin{align*}
V_4&=\ker\big(V\mapsto V(M')\big),\\
V_3\oplus V_4&=\ker\big(V\mapsto V(a_1\ \ldots\ a_n)\big),\\
V_2\oplus V_3\oplus V_4&=\ker\big(V\mapsto V(g_1\ \ldots\ g_n)\big).
\end{align*}
(So $\dim V_4=n-d'$, $\dim V_3=d_7+d_8$, $\dim V_2=d_4+d_5$, $\dim V_1=d_1$.)
Thus with $W$ determined by $G$ as above, $M'G$ meets the criteria of the lemma if and only if
the map
$V\mapsto V(g_1\ \ldots\ g_n)$ takes $W$ to a dimension 0 subspace,
$V\mapsto V(a_1\ \ldots\ a_n)$ takes $W$ to a dimension $d_4$ subspace,
$V\mapsto V(M')$ takes $W$ to a dimension $d_4+d_8$ subspace.

This splitting $V=V_1\oplus V_2\oplus V_3\oplus V_4$ 
corresponds to a splitting $V^{\#}=V_1'\oplus V_2'\oplus V_3'\oplus V_4'$
where $V_1'$, $V_1'\oplus V_2'$, $V_1'\oplus V_2'\oplus V_3'$ are uniquely determined
(recall that $V_4^{\perp}=V_1'\oplus V_2'\oplus V_3'$, etc.).
Let $(b_1\ \ldots\ b_n)$ be rows $d_1+d_4+d_5+1$ through $n-d_7-d_8$ of $N'$,
$(c_1\ \ldots\ c_n)$ be rows $d_1+d_4+d_5+1$ through $n-r$ of $N'$.

With $U'=W^{\perp}\subseteq V^{\#}$, $N'\,^tG^{-1}$ meets the criteria of the lemma
if and only if the map
$V^{\#}\mapsto V^{\#}(c_1\ \ldots\ c_n)$ takes $U'$ to a dimension 0 subspace,
and $V^{\#}\mapsto V^{\#}(b_1\ \ldots\ b_n)$ takes $U'$ to a dimension
$r-d_7-d_8$ subspace. 

Now we construct and count all dimension $n-j$ subspaces $W$ of $V$ so that the above
criteria for $W$ and $W^{\perp}$ is met.

We know by the symmetry of $M'\,^tN'$ that 
$$V_1'\oplus V_2'\oplus V_3'=\ker\big(V^{\#}\mapsto V^{\#}(b_1\ \ldots\ b_n)\big);$$
so $\ker\big(V^{\#}\mapsto V^{\#}(c_1\ \ldots\ c_n)\big)=V_1'\oplus V_2'\oplus V_3'\oplus U_4',$
with $U_4'\subseteq V_4'$.  We need to choose $W$ so that under the map
$V^{\#}\mapsto V^{\#}(N')$, $U'=W^{\perp}$ is mapped to a dimension 0 subspace.  Equivalently,
we need $U'\subseteq V_1'\oplus V_2'\oplus V_3'\oplus U_4'$, which means $W_4\subseteq W$
where $W_4=(V_1'\oplus V_2'\oplus V_3'\oplus U_4')^{\perp}\subseteq V_4$.
(So we can split $V_4=U_4\oplus W_4$.)  Since $(M',N')=1$, $(b_1\ \ldots\ b_n)\equiv(0\ D)\ (q)$
where $D$ is $(n-d')\times(n-d')$ with $\rank_qD=n-d'.$  So 
$V^{\#}\mapsto V^{\#}(b_1\ \ldots\ b_n)$ automatically takes $U_4'$ to a subspace of dimension
$r-d_7-d_8$.
Since $(M',N')=1$, we also know $\rank_q(c_1\ \ldots\ c_n)=n-j-d_4-d_8$.  Hence
$$\dim\ker\big(V^{\#}\mapsto V^{\#}(c_1\ \ldots\ c_n)\big)=r+d_1+d_4+d_5,$$
so $\dim U_4'=r-d_7-d_8$; thus $\dim W_4=\dim V_4-\dim U_4'=n-j-d_4-d_8.$

We need $\dim W=n-j$, and we need $V\mapsto V(a_1\ \ldots\ a_n)$ to take $W$ to
a dimension $d_4$ subspace.  Thus $W$ must be of the form $W_2\oplus W_3\oplus W_4$
where $W_3\oplus W_4\subseteq\ker\big(V\mapsto V(a_1\ \ldots\ a_n)\big)=V_3\oplus V_4$,
$\dim W_3\oplus W_4=n-j-d_4$, and $W_2$ is independent of $V_3\oplus V_4$.  Since
we need $V\mapsto V(M')$ to take $W$ to a dimension $d_4+d_8$ subspace, we must have
$W_3$ independent of $\ker\big(V\mapsto V(M')\big)=V_4$.  So we extend $W_4$ to
$W_3\oplus W_4$ where $\dim W_3=d_8$ with $W_3$ independent of $V_4$; thus we
have $\bbeta(d_7+d_8,d_8)q^{d_8(j-d_1-d_5-d_7)}$ choices for $W_3$.
Then we extend $W_3\oplus W_4$ to $W_2\oplus W_3\oplus W_4$ where $\dim W_2=d_4$
and $W_2$ is independent of $V_3\oplus V_4$; thus we have $\bbeta(d_4+d_5,d_4)q^{d_4(j-d_1-d_5)}$
choices for $W_2$.
\end{proof}

\begin{lem}\label{Lemma 6.6}  Suppose $\stufe$ is square-free, $\chi$ is a character modulo $\stufe$, and $q$
is a prime dividing $\stufe$.  Set $\F=\Z/q\Z$ and 
$$\sym_q^{\chi}(t)=\sum_{U\in\F^{t,t}_{\sym}}\chi_q(\det U).$$
Then $\sym_q^{\chi}(t)\not=0$ if and only if either (1) $\chi_q=1$, or (2) $\chi_q^2=1$ and $t$ is even.
\end{lem}

\begin{proof}  Say $q=2$.  Then $\chi_q=1$ (since $\stufe$ is square-free), so $\sym_q^{\chi}(t)$ is
the number of invertible, symmetric $t\times t$ matrices modulo $2$; clearly this is non-zero.

So suppose $q$ is odd.  Set $J=\begin{pmatrix} \omega\\&I_{t-1}\end{pmatrix}$ where $\omega$ is not
a square in $\F$.  
We know $GL_t(\F)$ acts by conjugation on the subset of invertible elements of $\F^{t,t}_{\sym}$; the orbits
are represented by $I$ and $J$.  
Note that for $U\in F^{t,t}_{\sym}$, $U$ is in the orbit of $I$ (resp. the orbit of $J$) if and only if, for some $\alpha\in\F^{\times}$,
we have
$\det U=\alpha^2$ (resp. $\det U=\alpha^2\omega$); 
also, given $\alpha\in\F^{\times}$, the
number of $U\in\F^{t,t}_{\sym}$ with $\det U=\alpha^2$ (resp. with $\det U=\alpha^2\omega$) is the number of
$U\in\F^{t,t}_{\sym}$ with $\det U=1$ (resp. with $\det U=\omega$).  
With $o(U)$ the order of the stabilizer of $U$, we know
$o(U)=o(U')$ when $U,U'$ are in the same orbit.
Hence
\begin{align*}
\sym_q^{\chi}(t)&=\frac{2}{q-1}\cdot\frac{\#GL_t(\F)}{o(I)}\cdot\frac{1}{2} \sum_{\alpha\in\F^{\times}}\chi_q(\alpha^2)\\
&\quad +
\frac{2}{q-1}\cdot\frac{\#GL_t(\F)}{o(J)}\cdot\frac{1}{2} \sum_{\alpha\in\F^{\times}}\chi_q(\alpha^2\omega)\\
&=\frac{\#GL_t(\F)}{q-1}\left(\frac{1}{o(I)}+\frac{\chi_q(\omega)}{o(J)}\right)\sum_{\alpha\in\F^{\times}}\chi_q^2(\alpha).
\end{align*}
Thus $\sym_q^{\chi}(t)=0$  if $\chi_q^2\not=1$.  Suppose $\chi_q^2=1$; then $\chi_q(\omega)=-1$ if and only if
$\chi_q\not=1$.  Also, by the theory of quadratic forms over finite fields (see, for instance, \cite{Ger}), we know
$o(I)=o(J)$ if and only if $t$ is odd, so the lemma follows.
\end{proof}

\begin{lem}\label{Lemma 6.7}
For $p$ prime, $t\in\Z_+$, we have $\sum_{\ell=0}^t \bbeta_p(t,\ell)\sym_p(\ell)=p^{t(t+1)/2}.$
\end{lem}

\begin{proof} Let $\F=\Z/p\Z$; take $V=\F x_1\oplus\cdots\oplus\F x_t$.  For each $t-\ell$-dimensional subspace $R$ of $V$,
fix $G_R\in GL_t(\F)$ so that $R=\F y_{\ell+1}\oplus\cdots\oplus\F y_t$ where $(y_1\ \cdots\ y_t)=(x_1\ \cdots\ x_t)G_R.$
Take $Q\in \F^{t,t}_{\sym}$ so that $\rank Q=\ell$.  Let $(V,Q)$ denote the quadratic space with $Q$ the quadratic form
on $V$ relative to the basis $(x_1\ \cdots\ x_t)$.  
By the uniqueness of the radical of $V$ (with respect to $Q$), there exists a unique $R$ so that
$^tG_RQG_R=\begin{pmatrix}U&0\\0&0\end{pmatrix}$ where $U\in\F^{\ell,\ell}_{\sym}$ with $U$ invertible, and there
are $\sym_p(\ell)$ possibilities for $U$ (depending on $Q$). 
Hence $\F^{t,t}_{\sym}$ is partitioned into sets
$\{Q:\ \rank Q=\ell\ \},$ $0\le \ell\le t$, and given $\ell$,
$\{Q:\ \rank Q=\ell\ \}$ is partitioned into sets
$\{Q:\ ^tG_RQG_R=\begin{pmatrix} U&0\\0&0\end{pmatrix}\ \}$,
$R$ varying over dimension $t-\ell$ subspaces of $V$, of which there are $\bbeta_p(t,t-\ell)=\bbeta_p(t,\ell)$,
$U$ varying over invertible elements on $\F^{\ell,\ell}_{\sym}$, of which there are $\sym_p(\ell).$
From this the lemma follows.
\end{proof}


\begin{thebibliography}{0}


\bibitem{B2} S. B\"ocherer, ``On the Hecke operator U(p)", with an appendix by Ralf Schmidt.
\textit{ J. Math. Kyoto Univ.} 45 (2005), no. 4, 807–829.

\bibitem{B3} S. B\"ocherer, ``On the space of Eisenstein series for $\Gamma_0(p)$: Fourier expansions", with an appendix by H. Katsurada.
\textit{ Comment. Math. Univ. St Pauli} 63 (2014), no. 1-2, 3-22.




\bibitem{E1} S.A. Evdokimov, ``A basis composed of eigenfunctions of Hecke operators in the theory of modular forms of genus $n$". (Russian) \textit{ mat. Sb. (N.S.)} 115(157) (1981), no. 3, 337-363.

\bibitem{E2} S.A. Evdokimov, Letter to the editors: ``A basis composed of eigenfunctions of Hecke operators in the theory of modular forms of genus $n$". (Russian) \textit{ mat. Sb. (N.S.)} 116(158) (1981), no. 4, 603.


\bibitem {F} E. Freitag, ``Siegel Eisenstein series of arbitrary level and theta series".
\textit{ Abh. Math. Sem. Univ. Hamburg } 66 (1996), 229-247.


\bibitem {Ger} L. Gerstein, ``Basic Quadratic Forms". \textit { Graduate
Studies in Math.} Vol. 90, Amer. Math. Soc., 2008.

\bibitem {HW} J.L. Hafner, L.H. Walling, ``Explicit action of Hecke operators
on Siegel modular forms". \textit{ J. Number Theory } 93 (2002), 34-57.





\bibitem {Klo} K. Klosin, ``A note on Hecke eigenvalues of hermitian Siegel Eisenstein series".
\textit{ Ramanujan J. } 35 (2014), no. 2, 287-298.









\bibitem {Ogg} A. Ogg, \emph{Modular forms and Dirichlet Series}, Benjaman Press, New York, 1969.


\bibitem{thetaI} L. Walling, ``Action of Hecke operators on Siegel theta series I".
\textit{ International J. of Number Theory } 2 (2006), 169-186.

\bibitem{thetaII} L. Walling, ``Action of Hecke operators on Siegel theta series II".
\textit{ International J. of Number Theory } 4 (2008), 981-1008.


\bibitem{Wal} L. Walling, ``Hecke eigenvalues and relations for degree 2 Siegel Eisenstein series".
\textit{ J. Number Theory } 132 (2013), 2700-2723.



\end{thebibliography}
\end{document}